\documentclass[a4paper]{article}
\usepackage{amsfonts,amsmath,amssymb,amsthm}
\usepackage{color}
\usepackage{ifpdf}
\ifpdf
    \usepackage[pdftex]{graphicx}
    \DeclareGraphicsExtensions{.png,.pdf,.jpg}
\else
   \usepackage[dvips]{graphicx}
   \DeclareGraphicsExtensions{.eps}
\fi
\graphicspath{{.}{figuresPL/}}

\newtheorem{lemma}{Lemma}[section]
\newtheorem{theorem}[lemma]{Theorem}
\newtheorem{corollary}[lemma]{Corollary}
\newtheorem{proposition}[lemma]{Proposition}
\newtheorem{remark}[lemma]{Remark}
\newtheorem{definition}[lemma]{Definition}

\def\Span{\textup{Span}}

\def\S{\mathbb{S}}
\def\R{\mathbb{R}}
\def\T{\mathbb{T}}
\def\E{\mathcal{E}}
\renewcommand{\S}{\mathbb{S}}
\def\Z{\mathbb{Z}}
\def\N{\mathbb{N}}

\def\O{\mathcal{O}}

\def\LM#1{\hbox{\vrule width.2pt \vbox to#1pt{\vfill \hrule width#1pt
height.2pt}}}
\def\LL{{\mathchoice {\>\LM7\>}{\>\LM7\>}{\,\LM5\,}{\,\LM{3.35}\,}}}
\def\restr{{\LL}}
\renewcommand{\phi}{\varphi}
\def\p{\phi}

\def\Div{\textup{div}\,}

\def\dist{\textup{dist}}
\def\ds{\displaystyle}

\def\HH{\mathcal{H}^{d-1}}
\def\1{\mathbf{1}}
\def\loc{\mathrm{loc}}
\def\Int{\mathrm{int}}
\def\CA{\mathcal{CA}}

\def\XXint#1#2#3{{\setbox0=\hbox{$#1{#2#3}{\int}$ }
\vcenter{\hbox{$#2#3$ }}\kern-.57\wd0}}

\def\X{X}
\def\eps{\varepsilon}
\renewcommand{\subset}{\subseteq}

\def\MHrF{M_{H_r}(F)}
\def\MHrFf{M_{H_r}(F+f)}
\def\MKF{M_{K_r}(F)}
\def\MKFf{M_{K_r}(F+f)}
\DeclareMathOperator*{\argmin}{\arg\!\min}

\begin{document}

\title{Plane-like minimizers and differentiability of the stable norm}
\author{A. Chambolle \footnote{CMAP, Ecole Polytechnique, CNRS,
        Palaiseau, France, email: antonin.chambolle@cmap.polytechnique.fr},
\and
        M. Goldman
        \footnote{Max Planck Institute for Mathematics in the Sciences, Leipzig, Germany, email: goldman@mis.mpg.de}
        \and M. Novaga
        \footnote{Dipartimento di Matematica, Universit\`a di Padova,
        via Trieste 63, 35121 Padova, Italy, email: novaga@math.unipd.it}
}
\date{}
\maketitle
\begin{abstract}
In this paper we 
investigate the strict convexity and the differentiability properties of the stable norm, which corresponds to the homogenized surface tension for a periodic perimeter homogenization problem (in a regular and uniformly elliptic case).
We prove that it is always differentiable in totally irrational directions,
while in other directions, it is differentiable if and only if the corresponding plane-like minimizers satisfying a strong Birkhoff property foliate the torus.
We also discuss the issue of the uniqueness of the correctors
for the corresponding homogenization problem.
\end{abstract}

\section{Introduction}

We wish to investigate here a conjecture raised by Caffarelli and De La Llave in \cite{CDLL} concerning the differentiability of the so-called stable norm (or minimal action functional) in geometric Weak KAM theory. When considering the area functional in the Euclidean space, it is a classical result that hyperplanes are minimizers under compact perturbations. Caffarelli and De La Llave \cite{CDLL} proved the existence of plane-like minimizers for more general integrands 
of the form
\[\int_{\partial^* E} F(x,\nu^E) \ d\HH(x) + \int_E g(x) \ dx\]
 where $F(x,\nu)$ is periodic in $x$, convex and one-homogeneous in $\nu$ and where $g$ is a periodic function. It is then possible to define the homogenized energy $\p(p)$ which represents the average energy of a plane-like minimizer in the direction $p$. Our aim is to study the regularity properties of this function $\p$.  Our main results are the following:
\begin{itemize}
\item If $p$ is ``totally irrational''  then $\nabla\p(p)$
exists.
\item The same occurs for any $p$ such that the plane-like minimizers satisfying the strong Birkhoff property (see Section \ref{birk}) give rise to a foliation of the space.
\item If there is a gap in this lamination and if
 $(q_1, \dots , q_k)\in \Z^d$ is a maximal
family of independent vectors such that $q_i\cdot p=0$, then
$\partial \phi(p)$ is a convex set of dimension $k$, and
$\phi$ is differentiable in the directions which are orthogonal
to $\{q_1, \dots,q_k\}$. In particular if $p$ is not totally irrational then $\p$ is not differentiable at $p$.
\item  $\phi^2$ is strictly convex.
\end{itemize}
We also discuss the uniqueness of the minimizers (see Theorem~\ref{thmane} and Appendix~\ref{SecGeneric}).

Our approach is based on the cell formula for $\p$, introduced by the first author and Thouroude in \cite{CT} (see \eqref{defphi}). This formula provides a characterization of $\p$ as the support function of some  convex set $C$ which gives the expression of its subgradient. Since $\p$ is a convex function, it is differentiable at a point $p$ if and only if $\partial \p(p)$ is reduced to a point. The set $C$ being made of the integral over the unit cell of the calibrations in the direction $p$, the problem of the differentiability of $\p$ reduces to the investigation of whether  two different calibrations in a given direction $p$ can have or not a different mean. The idea
 is to prove that every calibration calibrates every plane-like minimizer satisfying the strong Birkhoff property (see Section \ref{birk}) and is thus prescribed on the union of these sets which form a lamination of the space. If there is no gap in this lamination  (i.e. if it is a foliation) then the value of any calibration is fixed everywhere and thus its mean is also prescribed and $\p$ is differentiable at $p$. If $p$ is totally irrational, every gap of the lamination must have finite volume.
This can be used to show that the integral over a gap of two different calibrations must coincide, implying again the differentiability of $\p$. For non totally irrational vectors, the gap will be of infinite volume and using different heteroclinic minimizers inside the gap, one can show that there exists two different calibrations with different means, proving that $\p$ is non-differentiable. 

{ } One of the interesting features of our work is the connections it makes between the plane-like minimizers constructed in \cite{CT}
(as level sets of appropriate correctors in a periodic homogenization problem),
and the class of recurrent plane-like minimizers satisfying the strong Birkhoff property (see Section \ref{birk}). We also give a clear presentation of the notion of calibration, and clarify the structure of the subgradient of functionals with linear growth (see also \cite{andreu}), which could be of independent interest. Among the by-products of our analysis we obtain a proof of Ma\~ne's conjecture in this setting (see Theorem \ref{thmane} and 
Appendix~\ref{SecGeneric}).

{ } In \cite{auerbangert,HJG} the authors study the differentiability properties of the stable norm of the homology class $H_{d-1}(M,\R)$ for compact orientable Riemannian manifolds $M$. If $M$ is the $d$-dimensional torus with a Riemannian metric,  the stable norm is exactly  our function $\p$. In this case, our results generalize \cite{auerbangert,HJG} in the sense that we consider interfacial energies which do not necessarily come from a Riemannian metric, and we also allow the presence of a volume term. In some sense, a contribution of our work is also a simplification of some of the arguments of \cite{auerbangert,HJG}, which comes from the use of the cell formula \eqref{defphi}.
% , which in particular permits to consider only
% sets of finite perimeter, instead of rectifiable currents
% as in \cite{auerbangert,HJG}.

{ } The stable norm is a geometric analog of the minimal action functional of KAM theory whose differentiability has first been studied by Aubry and Mather \cite{aubry,mather} for geodesics on the two dimensional torus. The results of Aubry and Mather  have then been extended by Moser \cite{moser2}, in the framework of non-parametric integrands,  
and more recently by Senn and Bessi \cite{S, bessi}. In this context, the study of the set of non-selfintersecting minimizers, which correspond to our plane-like minimizers satisfying the Birkhoff property has been performed by Moser and Bangert \cite{moser,Bangertmin}, 
whereas the proof of the strict convexity of the minimal action has been recently shown by Senn \cite{Sennstrict}. Another related problem is the homogenization of periodic Riemannian metrics (geodesics are objects of dimension one whereas in our problem the hypersurfaces are of codimension one). We refer to \cite{burago,massart,BrButFra} for more information on this problem.

{ } The plan of the paper is the following. 
In Section \ref{secnotation} we recall some well known facts about functions of bounded variation, 
sets of finite perimeter, convex anisotropies, pairings between $BV$ functions and bounded vector fields, and introduce the notion of Class A minimizers, plane-like minimizers.
We define the stable norm, which is, in fact, the homogenized surface tension of a periodic homogenization problem and can be recovered with a cell formula.
In Section \ref{propPL} we give some general properties of Class A Minimizers. %In particular we prove an important  maximum principle. 
In Section \ref{seccalibre}  we define the calibrations and prove an important regularity result (Theorem \ref{zfixPL}). 
In Section \ref{birk} we introduce the Birkhoff property and prove that every calibration calibrates every plane-like minimizer which satisfies this strong Birkhoff property. In Section \ref{sechetero} we construct heteroclinic surfaces in every gap and use them in Section \ref{secpropphi} to study the strict convexity and differentiability properties of $\p$. In Section \ref{examples} we give some simple examples of energies for which the existence or absence of gaps can be proven. Finally in Section \ref{Gclosure} we show that the set of stable norms that can be attained starting from isotropic energies is dense in the set of all symmetric anisotropies.
Appendix~\ref{AppBir} gives an elementary proof of a separation result in $\Z^d$,
while in Appendix~\ref{SecGeneric} it is shown a generic uniqueness result
for the correctors minimizing the cell formula.
 
%\medskip 
\section*{Acknowledgments.}
The second author acknowledges very interesting discussions with G. Thouroude and B. Merlet. He also warmly thanks  Carnegie Mellon University for its hospitality during the finalization of this work.
The third author acknowledges partial support by the Fondazione CaRiPaRo Project
``Nonlinear Partial Differential Equations: models, analysis, and
control-theoretic problems''.
Some of the main ideas of this work have been clarified during a fruitful stay of the first two authors at the University of Padova.

\section{Notation and main assumptions}\label{secnotation}
\subsection{Basic notation}\label{notation}
We let $Q'=(0,1)^d$, $Q=[0,1)^d$ and $\T$ be the torus ($\R^d/\Z^d$). For $m\in \N$, we let $\mathcal{H}^m$ be the Hausdorff measure of dimension $m$. Given a vector $p\in \R^d$ we say that
\begin{itemize}
\item $p$ is totally irrational if there is no $q\in \Z^d\setminus \{0\}$ such that $p\cdot q=0$,
\item $p$ is not totally irrational if there exists such a $q\in \Z^d\setminus \{0\}$,
\item $p$ is rational if $p\in \R \cdot \Z^d$,
\item $p$ is irrational if it is not rational.
\end{itemize}
For a given $p\in \R^d$ which is not totally irrational we let $$\Gamma(p):=\{q\in \Z^d \,:\,  q\cdot p=0\}.$$ 
Then, there exists $(q_1, \dots, q_k)\in \Gamma(p)$ such that $\Gamma(p)=\Span_{\Z} (q_1, \dots, q_k)$. By a Gram-Schmidt procedure we can find $(\bar q_1, \dots, \bar q_k)\in \Gamma(p)$ such that $(\bar q_1, \dots, \bar q_k)$ is an orthogonal basis of $V^r(p):=\Span_{\R} (q_1, \dots, q_k)$. Notice that in general  $\Span_{\Z}(\bar q_1, \dots, \bar q_k)\neq \Gamma(p)$. We let $V^i(p):=
(\R p\oplus V^r(p))^\perp$ be the set of irrational directions,
that is, $V^i(p)\cap\Z^d=\{0\}$.

For $R>0$ we let $B_R$ be the open ball of radius $R$. Finally let $\S^{d-1}$ be the unit sphere of $\R^d$. In this paper we will take as a convention that all the sets are oriented by their inward normal.

\subsection{BV functions}
We briefly recall the definition of a function with bounded variation and a set of finite perimeter. For a complete presentation we refer to \cite{AFP,giusti}.

\begin{definition}
Let $\Omega$ be an open set of $\R^d$, we will say that a function $u\in L^1(\Omega)$ is a function of bounded variation if
\[\int_\Omega |Du|:= \sup_{\stackrel{z \in \mathcal{C}^1_c(\Omega)}{|z|_\infty \leq 1}} \int_{\Omega} u \ \Div z \ dx< +\infty.\]
We denote by $BV(\Omega)$ the set of functions of bounded variation in $\Omega$ (when $\Omega=\R^d$ we simply write $BV$ instead of $BV(\R^d)$). We define similarly the set $BV(\T)$ of periodic functions of bounded variation by 
choosing $\Omega=\T$ in the above definition.
%% replacing the set $\mathcal{C}^1_c(\Omega)$ by $\mathcal{C}^1(\T)$ in the definition of the total variation.

We say that a set $E\subset \R^d$ is of finite perimeter if its characteristic function $\chi_E$ has bounded variation. We denote its perimeter in an open set $\Omega$ by $P(E,\Omega):=\int_{\Omega} |D \chi_E|$, and write simply $P(E)$ when $\Omega=\R^d$.
\end{definition} 
{ } A function $v\in BV(\T)$ can equivalently be seen
as a $Q$-periodic function, with $\int_\T |Dv|=\int_{Q} |Dv|$
(which is also equal to $\int_{Q'}|Dv|$ iff $|Dv|(\partial Q)=0$,
and $\int_{Q+x}|Dv|$ for a.e.~$x$). We also let $BV_{\loc}$ be the set of functions of locally bounded variation. Similarly we will say that a set is of locally finite perimeter if its characteristic function is in $BV_{\rm loc}$. 

\begin{definition}
Let $E$ be a set of finite perimeter and let $t\in [0;1]$. We then define
\[E^{(t)}:= \left\{ x \in \R^d \,:\, \lim_{r\downarrow 0} \frac{|E\cap B_r(x)|}{|B_r(x)|}=t\right\}.\]
We denote by $\partial E:= \left(E^{(0)}\cup E^{(1)}\right)^c$ the measure theoretical boundary of $E$. We define the reduced boundary of $E$ by:
\[\partial^* E:= \left\{ x \in \textrm{Spt}( |D \chi_E|) \,: \, \nu^E(x):= \lim_{r \downarrow 0} \frac{ D \chi_E (B_r(x))}{|D \chi_E|(B_r(x))} \; \textrm{exists and } \; |\nu^E(x)|=1 \right\}\subset E^{\left(\frac 1 2\right)}.\] 
The vector $\nu^E(x)$ is the measure theoretical inward normal to the set $E$. When no confusion can be made, we simply denote $\nu^E$ by $\nu$.
\end{definition} 
\begin{proposition}
If $E$ is a set of finite perimeter then $D\chi_E=\nu\  \HH \LL\partial^* E$, $P(E)=\HH(\partial^* E)$ and $\HH(\partial E \setminus \partial^* E)=0$.
\end{proposition}
An important link between $BV$ functions and sets of finite perimeter is given by the Coarea Formula:

\begin{proposition}
Let $u\in BV(\Omega)$. For a.e.~$t\in\R$, $\{u>t\}$ has finite perimeter and
it holds
\[
|Du|(\Omega)=\int_\R \HH(\partial^*\{u>t\}\cap \Omega) \, dt.
\]
\end{proposition}

\subsection{Anisotropies}
Let $F(x,p):\R^d \times \R^d\to \R$ be a convex one-homogeneous function in the second variable such that there exists $c_0$ with
\[c_0 |p|\le F(x,p)\le \frac{1}{c_0} |p| \qquad \forall (x,p)\in \R^d\times \R^d.\]
We say that $F$ is elliptic if for some $\delta>0$, the function $F-\delta |p|$ is still a convex function. We denote by $W(x):=\{p \,:\, F(x,p)\le1 \}$ the unit ball of $F(x,\cdot)$ (when $F$ does not depend on the space variable $x$ we will just denote it by $W$). We define the polar function of $F$ by $$F^\circ(x,z):=\sup_{\{F(x,p)\le1\}} z\cdot p$$ so that $(F^\circ)^\circ=F$. If we denote by $F^*(x,z)$ the convex conjugate of $F$ with respect to the second variable then $\{F^*(x,z)=0\}=\{F^\circ \le 1\}$. If $F(x,\cdot)$ is differentiable then, for every $p\in \R^d$, 
$$F(x,p)=p\cdot \nabla_p F(x,p)$$ and
\[ z\in \{F^\circ(x,\cdot)\le 1\} \textrm{ with } p\cdot z=F(x,p)\ \Longleftrightarrow \  z=\nabla_p F(x,p).\]
 	If $F$ is elliptic and $\mathcal{C}^2(\R^d \times \R^d\setminus\{0\})$ then $F^\circ$ is also elliptic and $\mathcal{C}^2(\R^d \times \R^d\setminus\{0\})$. Moreover for every $x,y,z \in \R^d$, there holds 
\begin{equation}\label{strictconvF}F^2(x,y)-F^2(x,z)\ge 2\left(F(x,z)\nabla_p F(x,z)\right)\cdot (y-z) + C|y-z|^2\end{equation}
for some constant $C$ not depending on $x$ (and the same holds for $F^\circ$). Inequality \eqref{strictconvF} just state that $F^2$ is uniformly convex. 
We refer to \cite{schneider} for a proof of these results and much more about convex bodies. 

\subsection{Pairings between measures and bounded functions}
We fix in the following an elliptic  anisotropy $F$. Following Anzellotti \cite{Anzelotti} we define a generalized trace $[z,Du]$ 
for functions $u$ with bounded variation and bounded vector fields $z$ with divergence in $L^d$.
\begin{definition}
let $\Omega$ an open set with Lipschitz boundary, let $u\in BV(\Omega)$ and
let $z \in L^{\infty}(\Omega,\R^d)$ with $\Div z\in L^d(\Omega)$. 
We define the distribution $[z,Du]$ by 
 \[ \langle [z,Du], \psi \rangle= -\int_\Omega u \,\psi\, \Div z - \int_\Omega u \,  z \cdot \nabla \psi  
 \qquad \forall \psi \in \mathcal{C}^\infty_c(\Omega). 
 \]
\end{definition}

{ } If $u\in BV(\T)$ and $z \in L^{\infty}(\T,\R^d)$, with $\Div z\in L^d(\T)$, we can easily define the distribution $[z,Du]$ in a similar way.

 \begin{theorem}
 The distribution $[z,Du]$ is a bounded Radon measure on $\Omega$ and if $\nu$ is the inward unit normal to $\Omega$, there exists a function $[z,\nu]\in L^\infty(\partial \Omega)$ such that the generalized Green's formula holds,
 \[ \int_\Omega [z,Du]=- \int_\Omega u \Div z -\int_{\partial \Omega} [z, \nu] u\,d\mathcal H^{d-1}.\]
 The function $[z,\nu]$ is the generalized (inward) normal trace of $z$ on $\partial \Omega$. If $u\in BV(\T)$ 
 and $z \in L^{\infty}(\T,\R^d)$, with $\Div z\in L^d(\T)$, there holds
 \[\int_{\T} [z,Du]=-\int_{\T} u \Div z.\]
  \end{theorem} 
  
{ } Given $z \in L^{\infty}(\T,\R^d)$, with $\Div z\in L^d(\T)$, we can
also define the generalized trace of $z$ on $\partial E$, where $E$ is a set of locally finite perimeter. 
Indeed, for every bounded open set $\Omega$ with smooth boundary, we can define as above the measure $[z,D\chi_E]$ on $\Omega$. 
Since this measure is absolutely continuous with respect to $|D\chi_E|= \HH \restr \partial^* E$ we have 
\[
[z,D\chi_E]= \psi_z(x) \HH\restr \partial^* E
\]
with $\psi_z \in L^{\infty}(\partial^* E;\HH)$ independent of $\Omega$. 
We denote by $[z,\nu]:=\psi_z$ the generalized (inward) normal trace of $z$ on $\partial E$. 
If $E$ is a bounded set of finite perimeter, by taking $\Omega$ stricly containing $E$,
we have the generalized Gauss-Green Formula 
\[\int_E \Div z =-\int_{\partial^* E} [z,\nu] d \HH.\]
We notice that there has been a lot of interest in  defining the trace of bounded vector fields with divergence a bounded measure, 
on boundaries of sets of finite perimeters \cite{ACM,CTZ}, since this is related to the study of conservations laws.

\subsection{Plane-like minimizers and the stable norm}

Given a set $E$ of locally finite perimeter, we consider the energy
\begin{equation}\label{MainEnergy}
\E(E,A) \ :=\ \int_{\partial^* E\cap A} F(x,\nu^E)\,d\HH+\int_{E\cap A}g(x)\,dx
\end{equation}
with $F,g$ $Q$-periodic continuous functions, and $\int_Q g=0$. Here $\nu^E$ is the inner normal to $\partial^*E$,
so that $\int_{\partial^*E\cap A}F(x,\nu^E)\,d\HH=\int_A F(x,D\chi_E)$. We also assume that
$F$ is convex, one-homogeneous with respect to its second variable,
and satisfies for some $c_0>0$
\begin{equation}\label{boundF}
c_0|p|\le F(x,p)\le c_0^{-1}|p|
\end{equation}
for any $(x,p)\in\R^d\times\R^d$. A fundamental assumption throughout the
paper is that the energy is coercive, in the sense that
\begin{equation}\label{hypcoerF}
\E(E,Q)\ \ge\ \delta P(E,Q)
\end{equation}
for some $\delta>0$ independent of $E$.
This will be ensured if~\eqref{boundF} holds and $g$ is small enough in some
appropriate norm.
When $F(x,p)=|p|$ and $g=0$, the energy $\E$ is just the perimeter. In that case, it is well known that planes are minimizers under compact perturbations. 
In addition, the Bernstein Theorem states that, if $d\le 7$, the only  minimizers of the perimeter under compact perturbations are the hyperplanes (see \cite{giusti}). In \cite{CDLL}, Caffarelli and De la Llave proved that, for general energies $\E$, even if hyperplanes are not  minimizers anymore there still exist plane-like minimizers.

\begin{definition}
We say that a set $E$ of locally finite perimeter
is a Class A Minimizer of $\E$ if, for any $R>0$, the set $E$ minimizes $\E(E,B_R)$ under compact perturbation in $B_R$.
\end{definition} 

\begin{theorem}[\cite{CDLL}]\label{thmcdll}
There exists $M>0$ depending only on $c_0$ and $\delta$ such that for every $p\in \R^d\setminus\{0\}$ and $a\in \R$, there exists a Class A Minimizer $E$ of $\E$ such that  
\begin{equation}\label{PLexist}
%\{x\cdot p> a+ M|p|\}\subset 
\left\{ x\cdot \frac{p}{|p|} > a+M\right\}
\ \subset \ E\ \subset \ \left\{ x\cdot \frac{p}{|p|} > a-M\right\}.
\end{equation}
Moreover $\partial E$ is connected.
\end{theorem}

\begin{definition}
If $E$ satisfies \eqref{PLexist} for some $M>0$, we say that $E$ is a plane-like set.
If $E$ is a Class A Minimizer of $\E$ satisfying \eqref{PLexist} we say that $E$ is a plane-like minimizer.
\end{definition}

% \begin{remark}\rm In dimension two, we can easily show (using the fixed width
% of the  plane-like minimizers) that any Class A Minimizer must be plane-like. An interesting question would be to know if this Bernstein type of results holds for $d\le 4$.\problem{DO we realy know how to prove the d=2?}
% \end{remark}

{ } The existence of Class A Minimizers is closely related to the existence of minimizers of functionals of the form 
\[
\int_{\R^d} G(x,u,\nabla u)
\]
satisfying $\sup |u(x)-p\cdot x| <+\infty$ for some $p$ (the vector $p$ is often called the rotation vector) in Weak KAM theory (see \cite{moser,S}). The analogous of the minimal action functional of Weak KAM theory in our setting is the so-called stable norm introduced by Federer.
\begin{definition}
Let $p\in \R^d\setminus\{0\}$ and let $E$ be a plane-like minimizer of $\E$ in the direction $p$. We set  
\[\p(p):= |p|\lim_{R\to \infty} \frac{1}{\omega_{d-1} R^{d-1}} \, \E(E, B_R),\]
where $\omega_{d-1}$ is the volume of the unit ball in $\R^{d-1}$.
\end{definition}
{ } Caffarelli and De La Llave  proved that this limit exists and does not depend on $E$. In \cite{CT}, the first author and Thouroude related this
definition to the cell formula: 
\begin{equation}\label{defphi}
\p(p)\ =\ \min\left\{
\int_{\T} F(x,p+D v(x))\,+\int_{\T}  g(x)(v(x)+p\cdot x)\,dx\,:\,
v\in BV(\T)
\right\}\,,
\end{equation}
where the measure $F(x,p+D v)$ is defined for $v\in BV(\T)$ by $F(x,p+D v):=F(x,\frac{p+Dv}{|p+Dv|})|p+Dv|$. It is obvious from~\eqref{defphi} that $\p$ is a convex, one-homogeneous function. It is also shown in~\cite{CT}
that the minimizers of \eqref{defphi} give an easy way to construct plane-like minimizers:
\begin{proposition}[\cite{CT}]
Let $v_p$ be a minimizer of \eqref{defphi} then for every $s\in \R$, the set $\{v_p(x)+ p\cdot x >s\}$ is a plane-like minimizer of $\E$ in the direction $p$.
\end{proposition}
{ } We will make the following additional hypotheses on $F,g$:
\begin{itemize}
\item $F$ is $\mathcal{C}^{2,\alpha}(\R^d\times (\R^d\setminus\{0\}))$, and $g\in \mathcal{C}^{1,\alpha}(\R^d)$,
\item $F$ is elliptic (that is $F(x,p)-C|p|$ is a convex function of $p$)
\end{itemize}
Under these assumptions, one can
show that there exists a periodic vector field $\sigma$ of class $\mathcal{C}^{2,\alpha}$, with
$\Div \sigma=g$ and such that
$F'(x,p)=F(x,p)-\sigma(x)\cdot p \ge c_0'|p|$, for some $c_0'>0$.
The proof follows the same idea as in~\cite{CT} (the only difference
is that thanks to our regularity assumptions we need not
rely on~\cite{BourgainBrezis}).
From~\eqref{hypcoerF} and~\eqref{boundF},
it follows (see for example \cite{CT}) that $\int_\T F(x,Dv)+\int_\T (1+\eps) g v\,dx\ge  \delta/2
\int_\T|Dv|$ if $\eps>0$ is small enough. Hence,
$(1+\eps)g \in \partial H(0)$, the subgradient at zero of 
$H(v) := \int_\T F(x,Dv)$. This is a $1$-homogeneous, convex, l.s.c.~functional
defined on $L^p(\T)$ for any $p\in [1,\infty]$ (letting $H(v)=+\infty$
for $v\not\in BV(\T)$), which is the support function of
\[
K_H=\{ \Div\sigma\,:\, \sigma\in \mathcal{C}^\infty(\T;\R^d)\,, F^\circ(x,\sigma(x))\le 1
\ \forall \,x\in\T\}.
\]
From Hahn-Banach's theorem, one deduces that $(1+\eps)g$ is in the
closure of $K_H$ for the topology $(L^{p'},L^p)$. For $p<\infty$, large,
it coincides with the strong closure, hence $(1+\eps)g$
is the $L^p$-limit of a sequence
$\Div\sigma_n$, with $\sigma_n$ smooth and $F^\circ(x,\sigma_n(x))\le 1$.
We solve then $\Delta u_n=g-\Div (\sigma_n/(1+\eps))$ in the torus.
On one hand, $u_n\in \mathcal{C}^{3,\alpha}(\T)$ by elliptic regularity.
On the other hand,
$\|u_n\|_{2,p}\le C\|g-\Div(\sigma_n/(1+\eps))\|_{p}$, which goes to zero with $n$.
In particular, $\|\nabla u_n\|_\infty$ is arbitrarily small.
We choose $n$ so large that this quantity is less than $c_0\eps/(2+2\eps)$,
and let $\sigma=\sigma_n/(1+\eps)+\nabla u_n\in \mathcal{C}^{2,\alpha}(\T)$.
Then for any $p$,
$\sigma(x)\cdot p \le F(x,p)/(1+\eps)+\nabla u_n\cdot p 
\le F(x,p)-c_0 \eps/(1+\eps)|p|+c_0\eps|p|/(2+2\eps)$,
so that the claim holds with $c'_0= c_0\eps/(2+2\eps)>0$.

For this reason, we can replace without loss of generality
$F$ with $F'$ and $g$ with zero in~\eqref{MainEnergy}
without changing anything to the problem.
To simplify the notation we will therefore assume that $g=0$ in the rest of the paper.\footnote{The hypothesis $g\in \mathcal{C}^{1,\alpha}$ could
be relaxed to $g$ Lipschitz. Indeed, the regularity hypothesis on $F$ is mainly there to ensure that regularity theory and maximum principle hold for the plane-like minimizers (see Proposition \ref{maxprinc}).}

In the following we let
\[
\X:=\{ z \in L^{\infty}(\T,\R^d) \,:\, \Div z=0 \ , \ F^\circ(x,z(x))\le 1 \ a.e \}.
\]
We remark that $\X$ is closed (hence compact)
for the $L^\infty$ weak-$*$ topology.
Indeed if $z_n\in X$, $z_n\stackrel{*}{\rightharpoonup} z$ one
sees that for any $p\in \mathbb{Q}^d$,
the average of $z\cdot p$ in any ball $B_\rho(x)$ is less than the
average of $F(\cdot,p)$ (since it is true for $z_n$). For a.e.~$x$ it follows
that $z(x)\cdot p\le F(x,p)$ for all $p\in \mathbb{Q}^d$ (hence $\R^d$),
that is, $F^\circ(x,z(x))\le 1$.

The  following characterization of the  subdifferential of one-homogeneous functionals is classical and readily follows for
 example from the representation formula \cite[(4.19)]{DMBou} (see also \cite[Prop 3.1]{CT}.
\begin{proposition}
 A function $v\in BV(\T)$ is a minimizer of \eqref{defphi} if and only if there exists $z\in X$ such that 
\[ [z,Dv+p]= F(x,Dv+p).\]
\end{proposition}
The next result is the starting point of our analysis on the differentiability properties of $\p$.
\begin{proposition}\label{propdual}
 There holds
%\[\p (p)=\sup_{z \in \X} \left(\int_{\T} z\right)\cdot p +\int_{\T} g x \cdot p  \]
\[\p (p)=\sup_{z \in \X} \left(\int_{\T} z\right)\cdot p\,.\]
\end{proposition}

\begin{proof} This is a standard convex duality result.
 Notice that the $\sup$ on the right-hand side is in fact
a $\max$, since $X$ is compact.
Now, for every $v\in BV(\T)$ and every $z \in \X$, 
\[\begin{array}{rll}
  \ds \int_{\T} [Dv+p,z] &\le \ds \int_{\T} F(x,Dv+p) & \textrm{ hence}\\
 \ds -\int_{\T} v \, \Div z +\left( \int_{\T} z \right) \cdot p &\le \ds \int_{\T} F(x,Dv+p) &\textrm{ and since } \Div z=0,\\[8pt]
 \ds \left( \int_{\T} z \right) \cdot p & \le \ds \int_{\T} F(x,Dv+p) & \\[8pt]
 \ds \left( \int_{\T} z \right) \cdot p  & \le \ds\p(p) &  \textrm{ taking the infimum on $v$, thus}\\[8pt]
 \ds \max_{z \in \X} \left(\int_{\T} z\right)\cdot p &\le \ds \p(p) \,.
\end{array}\]
To prove the opposite inequality let $v_p\in BV(\T)$ be a minimizer in the definition of $\p(p)$ and let $z\in X$ be such that
$$
\ds \int_{\T} [Dv_p+p,z] = \ds \int_{\T} F(x,Dv_p+p)\,,
$$
then
\[
\ds \p(p) = \ds \int_{\T} F(x,Dv_p+p)\le \ds \max_{z} \left(\int_{\T} z \right) \cdot p.
\] 
%  \ds \int_{\T} [Dv_p+p,z] &= \ds \int_{\T} F(x,Dv_p+p) & \textrm{hence}\\
%\ds \int_{\T} F(x,Dv_p+p)& \ds =\left(\int_{\T} z \right) \cdot p & \textrm{from which}\\
%\ds \p(p) &\le \ds \sup_{z} \left(\int_{\T} z \right) \cdot p.
%\end{array}\]
\end{proof}

{ }  Proposition \ref{propdual} shows that $\p$ is the support function of the convex set
\begin{equation}\label{defC}
C\ : =\ \left\{
\int_{Q} z(x)\,dx\,:\, z\in \X\right\}\,,
\end{equation}
so that 
\[\p(p)=\max_{\xi \in C} \ \xi\cdot p.\]
Observe that $C$ is trivially compact in $\R^d$, being $X$ a compact set.

{ } The subgradient of $\p$ at $p\in\R^d$ is
given by
\begin{equation}
\partial\p (p)\ =\ \left\{ \xi\in C\,:\, \xi\cdot p=\p(p)\right\}.
\end{equation}
Any $\xi\in \partial \p(p)$ is associated to a field $z$ as
in \eqref{defC}. We will exploit the following fact:
\[\p \textrm{ is differentiable at } p \quad  \Longleftrightarrow \quad \partial \p (p) \textrm{ is a singleton}.\]

\section{Properties of Class A Minimizers}\label{propPL}
We first start by recalling some well known facts about Class A and plane-like minimizers~\cite{ASS,CDLL,CT}.

\begin{proposition}\label{propdensity}
Let $E$ be a Class A Minimizer.
Then the reduced boundary $\partial^* E$ is of class $C^{2,\alpha}$ and $\mathcal{H}^{d-3} (\partial E \setminus \partial^* E)=0$. Moreover, there exists $\gamma>0$ and $\beta>0$ such that
\begin{itemize}
\item if $x\in \overline{E}$ then $|B_r(x)\cap  E|\ge \gamma r^d$ for every $r> 0$,
\item if $x\in E^c$ then $|B_r(x)\setminus  E|\ge \gamma r^d$ for every $r>0$,
\item if $x\in \partial E$ then $\beta r^{d-1}\le |D\chi_E|(B_r(x))\le \frac{1}{\beta} r^{d-1}$ for every $r>0$.
\end{itemize}
As a consequence, we will assume in the following that our Class A Minimizers
are all open sets (indeed, we can identify them with their points of
density 1, which clearly are an open set because of the second density
estimate). Moreover, the  topological boundary $\partial E$
agrees with the measure theoretical boundary of $E$.
\end{proposition}

{ } The stability of plane-like minimizers under convergence is a crucial point in the theory.
\begin{proposition}\label{stabPL}
 Let $E_n$ be a sequence of plane-like minimizers satisfying \eqref{PLexist} with a uniform $M$ and converging in the $L^1_{\loc}$ topology to a set $E$, then $E$ is also a plane-like minimizer. Moreover, $E_n\to\bar E$, $E_n^c \to E^c$, and of $\partial E_n \to \partial E$ in the Kuratowski sense.
\end{proposition}
\begin{proof}
 The stability of the plane-like minimizers is a well known fact~\cite[Section 9]{CDLL}. The Kuratowski (or local Hausdorff)
convergence easily follows from the uniform density estimates for plane-like minimizers (Proposition~\ref{propdensity}).
Indeed, let $\eps>0$ be fixed and let $x \in E \cap \left\{y \,:\, d(y,\partial E)> \eps \right\}$.
If $x$ is not in $E_n$ then by the density estimates we have
\[
|E_n \Delta E|\geq |B_\eps(x)\setminus E_n|\geq \gamma \eps^d.
\] 
This is impossible if $n$ is big enough because $|E_n\Delta E|$ tends to zero. Similarly, we can show that for $n$ big enough, all the points of $E^c\cap\left\{y \,:\, d(y,\partial E)> \eps \right\}$ are outside $E_n$. This shows that  $\partial E_n \subset \left\{y \,:\, d(y,\partial E)\le \eps \right\}$. Inverting $E_n$ and $E$, the same argument proves that $\partial E \subset \left\{y \,:\, d(y,\partial E_n)\le \eps \right\}$ giving the Kuratowski convergence of $\partial E_n$ to $\partial E$.
\end{proof}

{ } Another simple (and classical) consequence of the density estimates is the following
\begin{proposition}\label{supDu}
 Let $u\in BV_{\loc}(\R^d)$ then for every $R>0$,
\[
Spt(|Du|)\cap B_R\ =\ B_R\cap 
\overline{\bigcup_{s}\partial^*\{u>s\}} .
\]
where in the union we consider only the levels for which
$\{u>s\}$ has finite perimeter in $B_R$.

{ } If in addition $v_p$ is a minimizer of \eqref{defphi} and $u(x)=v_p(x)+p\cdot x$, then 
\begin{itemize}
\item $P(\{u>s\}\cap B_R)<+\infty$ for every $s\in \R$;
\item $\overline{\partial^*\{u>s\}}=\partial \{u>s\}$; 
\item the function $u^+$ is u.s.c., and $u^-$ is l.s.c.; % are usc respectively lsc [NE SERT A RIEN, SINON DEFINIR $u^\pm$...?]
\item $Spt(|Du|)\cap B_R=\left(\bigcup_{s \in \R} \partial \{u>s\} \cap B_R\right)\cup \left(\bigcup_{s \in \R} \partial \{u\ge s\} \cap B_R\right)$.
\end{itemize}
\end{proposition}
Here, $u^\pm(x)$ are classicaly defined as the approximate upper and lower
limits of $u$ at $x$, see~\cite{AFP}.

\begin{proof}
Let us show the first assertion. For $u\in BV_{\loc}(\R^d)$, if $x \notin Spt(|Du|)$, then there exists $\rho>0$ with $B_\rho(x) \subset \{Du=0\}$ and thus $u$ is constant on $B_\rho(x)$, which implies that 
$$
x \notin \overline{\bigcup_{s}%% \,:\, P(\{u>s\}; B_R)<+\infty}
 \partial^*\{u>s\} \cap B_R}.
$$ 
On the contrary, if $x \in Spt(|Du|)$ then for every $\rho>0$, by the Coarea Formula,
\[|Du|(B_\rho(x))=\int_\R \HH(\partial^*\{u>s\} \cap B_\rho(x)) \ ds >0\]
thus for every $\rho>0$ there exists $x_\rho \in B_\rho(x)\cap \partial^*\{u>s_\rho\}$ for some $s_\rho$ since $x_\rho$ tends to $x$ when $\rho\to 0$, this proves the other inclusion.\\
Given a minimizer $v_p$, the other properties follow easily from the density estimates.
\end{proof}

%\begin{proposition}\emph{\cite[Prop. 7.3]{CDLL}}
%Let $E$ be a Class A plane-like minimizer,
%then $E$ and $E^c$ are connected sets.
%\end{proposition}

{ } The following maximum principle for minimizers is a cornerstone of the theory.
\begin{proposition}\label{maxprinc}
%For any 
%%plane-like
%Class A Minimizer $E$, the reduced boundary $\partial^* E$ is of class $C^{2,\alpha}$ and $\mathcal{H}^{d-3} (\partial E %\setminus \partial^* E)=0$.
Let $E_1\subset E_2$ be two Class A Minimizers with connected boundary, 
then $\mathcal{H}^{d-3} (\partial E_1 \cap \partial E_2)=0$.
\end{proposition}

\begin{proof}
%By classical regularity theory for almost minimizers of elliptic variational integrands  \cite{ASS,Duzaar}, if $F$ is $C^{2,\alpha}(\R^d\times \S^{d-1})$, then  for every Class A
%Minimizer $E$, the reduced boundary $\partial^*E$ is of class $C^{1,\frac{\alpha}{2}}$. Using classical incremental ratios arguments, we get that it is locally a graph  of a $W^{2,q}$ function for every $q>0$. Moreover, 
%For $(y,r,p)\in \R^{d-1}\times\R\times\R^{d-1}$ we let 
%\[\tilde F(y,r,p)=F((y,r), (-p,1)).\]
%Then if  $\partial^*E$ is (locally) the graph of $u \in C^{2,\alpha}(\Omega)$, $u$ (locally) minimizes the functional
%%\[\int_\Omega \tilde F(x,u,\nabla u) dx + \int_\Omega \int_0^{u(x)} g(x,t) dt \ dx.\]
%\[\int_\Omega \tilde F(y,u,\nabla u) \ dy.\]
%The function $u$ is then a solution of the elliptic PDE with H\"older continuous coefficients
%%\begin{equation}\label{ELu}\frac{\partial \tilde F}{\partial r}(x,u,\nabla u)-\Div[\nabla_p \tilde F(x,u,\nabla u)]= g(x,u).\end{equation}
%\begin{equation}\label{ELu}
%\frac{\partial \tilde F}{\partial r}(y,u,\nabla u)-\Div[\nabla_p \tilde F(y,u,\nabla u)]=0.
%\end{equation}
%By Schauder Theory \cite[Thm 6.2]{GT}, $u$ is then $C^{2,\alpha}$ and by \cite[Section II.3]{ASS}, $\mathcal{H}^{d-3}(\partial E \setminus \partial^* E)=0$.

%Let $E_1\subset E_2$ be two different Class A %plane-like Minimizers. 
 We shall now prove that $\partial^*E_1\cap \partial^* E_2 =\emptyset$. Let us assume by contradiction that $\partial^*E_1\cap \partial^* E_2 \neq \emptyset$ then we can find $\bar x \in \partial^*E_1\cap \partial^* E_2$ such that $\partial^*E_1\cap B_r(\bar x) \neq \partial^* E_2 \cap B_r(\bar x)$ for every $r>0$.
 Since $E_1\subset E_2$, $E_1$ and $E_2$ have the same tangent space at $\bar x$
and they can be seen as graphs over the same domain $D$ of two functions
$v_1, v_2 \in C^{2,\alpha}(D)$, with $v_2\ge v_1$. % satisfying \eqref{ELu}.
For $(y,r,p)\in \R^{d-1}\times\R\times\R^{d-1}$, let 
\[\widetilde F(y,r,p):=F((y,r), (-p,1)).\]
The functions $v_i \in C^{2,\alpha}(D)$, $i=1,2$,  (locally) minimize the functional
%\[\int_\Omega \widetilde F(x,u,\nabla u) dx + \int_\Omega \int_0^{u(x)} g(x,t) dt \ dx.\]
\[\int_D \widetilde F(y,u,\nabla u) \ dy\]
and thus solve the elliptic PDE with H\"older continuous coefficients
%\begin{equation}\label{ELu}\frac{\partial \widetilde F}{\partial r}(x,u,\nabla u)-\Div[\nabla_p \widetilde F(x,u,\nabla u)]= g(x,u).\end{equation}
\begin{equation}\label{ELu}
\frac{\partial \widetilde F}{\partial r}(y,v_i,\nabla v_i)-\Div[\nabla_p \widetilde F(y,v_i,\nabla v_i)]=0.
\end{equation}
Consider the function $w=v_2-v_1$. 
Up to reducing the domain, we can assume that $\bar x\in \partial D$,
 $w> 0$ in $D$ and $w(\bar x)=0$. We must then have $\nabla w(\bar x)=0$. Let 
 \begin{align*} A(x)&:= \int_0^1 \nabla^2_p \widetilde F( x, v_2(x), \nabla v_2(x) -\nabla v_1(x)) dt\\
B_1(x)&:= \int_0^1 \nabla_p \frac{\partial }{\partial r} \widetilde F (x, t v_2(x)+(1-t)v_1(x)), \nabla v_2(x)) dt \\
 B_2(x)&:= \int_0^1 \nabla_p \frac{\partial }{\partial r} \widetilde F (x, v_2(x), t\nabla v_2(x)+(1-t) \nabla v_1(x)) dt \\
 c_1(x)&:=\int_0^1 \frac{\partial^2}{\partial r^2} \widetilde F(x, t v_2+(1-t)v_1, \nabla v_2) dt \\
 %c_2(x)&:= -\int_0^1 g(x, t v_2 +(1-t) v_1) dt 
\end{align*}
then $w$ satisfies the linear non-degenerate elliptic PDE,
%\[-\Div (A(x) \nabla^2 w)- \Div( B_1(x) w) +B_2(x)\cdot \nabla w +c_1(x) w +c_2(x) w =0.\]
\[-\Div (A(x) \nabla^2 w)- \Div( B_1(x) w) +B_2(x)\cdot \nabla w +c_1(x) w  =0.\]
By Hopf's Lemma \cite[Lemma 3.4]{GT}, this implies that
 $\nabla w(\bar x)\cdot\nu_D < 0$, which gives a contradiction. 
Thus $\partial E_1$ and $\partial E_2$ can only intersects in singular points which are of $(d-3)$-Hausdorff measure zero~\cite{ASS}.
\end{proof}
\begin{remark}\rm
 In the case of isotropic functionals i.e. $F(x,\nu)=a(x) |\nu|$, \cite{Simon} shows that in fact two %plane-like minimizers, 
minimizers which are contained one in the other cannot touch at all.
\end{remark}

\begin{proposition}\label{bern}
Let $d=2$ and let $E$ be a Class A Minimizer, then $E$ is a plane-like minimizer and $\partial E$ is connected.  
\end{proposition}
\begin{proof}
By Proposition \ref{maxprinc}, $\partial E$ is of class $C^{2,\alpha}$ and is composed of a locally finite union of curves. 
Moreover each of these curves has infinite lenght by the minimality of $E$. Let $\gamma$ be such a curve and $H$ be a plane-like minimizer with connected boundary. Due to the regularity and minimality of $E$ and $H$, 
$\partial H$ can intersect $\gamma$ in at most one point. Since by Theorem \ref{thmcdll} there exist such plane-like minimizers inside every strip of width $M$, it follows that $\gamma$ is included in such a strip. Using again the minimality of $E$, we then get that $\partial E$ is connected and $E$ is plane-like.   
\end{proof}
Proposition \ref{bern} is reminiscent of Bernstein Theorem, and
the same result also holds for $d=3$ under the additional assumption that $F$ does not depend on $x$ \cite{white}. 
However, in \cite{exMorgan} it is shown 
that it is non longer true in four dimensions, even for a function $F$ independent of $x$.

\section{Calibrations}\label{seccalibre}
We now introduce the notion of calibration. 
\begin{definition}
 We say that a vector field $z \in \X$ is a periodic calibration of a set $E$ of locally finite perimeter if, for every open set $A$, we have
\[\int_A [z,D \chi_E]=\int_A F(x,D \chi_E).\]
When no confusion can be made, by calibration we mean a periodic calibration.
\end{definition}

{ } The constant interest towards calibrations in the study of minimal surfaces, comes from the following result:
\begin{theorem}\label{calibregalPL}
 If $E$ is a set for which there exists a calibration (not necessarily periodic),  then $E$ is a Class A Minimizer.
\end{theorem}
\begin{proof}
It follows by integration by parts, using $\Div z=0$.
\end{proof}
{ } Calibrations are very stable objects, as shown by the following Proposition:

\begin{proposition}\label{stabilitycalib}
 Let $E_n$ be %plane-like
Class A minimizers converging in the $L^1_{\loc}$-topology to a set $E$.
Assume that the sets $E_n$ are calibrated by $z_n \in \X$ and that $z_n$ converges to a field $z$ weakly-$*$ in $L^\infty$. Then $z$ calibrates $E$ (which is
thus also a minimizer).
\end{proposition}

\begin{proof}
Let $\phi\in \mathcal{C}_c^\infty(\R^d)$, $\phi\ge 0$. Observe
that since $X$ is compact, $z\in X$. Hence
\begin{multline*}
\int_{\partial^* E_n} \phi F(x,\nu^{E_n})\,d\HH\ =\ \int_{\R^d} \phi [z_n,D\chi_{E_n}]
\ =\ -\int_{E_n} z_n\cdot\nabla \phi 
\\  \stackrel{n\to\infty}{\longrightarrow}\ -\int_{E} z\cdot\nabla \phi 
\ =\ \int_{\R^d} \phi [z,D\chi_{E}]
\ \le\ \int_{\partial^* E} \phi F(x,\nu^{E})\,d\HH
\end{multline*}
and the reverse inequality follows by lower-semicontinuity of
the total variation. Hence $z$ is a calibration for $E$.

\end{proof}

{ } A natural way of producing calibrations for a set is through the cell problem \eqref{defphi}.
By Proposition \ref{propdual} there exists $z\in X$ such that, for any minimizer $v_p$ of~\eqref{defphi},
\begin{equation}
\int_\T [z, Dv_p+p]\ =\ \int_\T F(x,p+Dv_p)\ =\ \p(p)\,.
\end{equation}
%showing that $z$ is a ``calibration'' for problem~\eqref{defphi}.
We say that such a vector field $z$ is a calibration in the direction $p$.
Conversely, if $z\in X$ and $v\in BV(\T)$ are such that
\begin{equation}
\int_\T [z, Dv+p]\ =\  \int_\T F(x,p+Dv)\,,
\end{equation}
then $v$ is a minimizer of~\eqref{defphi}
%and $\int_\T z\,dx\in\partial \phi(p)$, so that 
and $z$ is a calibration in the direction $p$. Repeating almost verbatim the proof of the Coarea Formula \cite[Th. 3.40]{AFP}, there holds,

\begin{proposition}
 Let $A\subset \R^d$ be an open set. For $u\in BV(A)$, letting for $s\in \R$, $E_s:=\{u>s\}$, there holds
\[\int_{\R} \int_A F(x, D \chi_{E_s})  \ ds= \int_A F(x, D u).\]
\end{proposition}

We then deduce,

\begin{proposition}\label{calibrEs}
 Let $z\in \X$ be a calibration in the direction $p$,  then for every minimizer $v_p$ of \eqref{defphi} and every $s \in \R$, $z$ calibrates the set $E_s := \{ v_p + p\cdot x > s\}$. Conversely for $v \in BV(\T)$, if $z\in \X$ calibrates  all the sets $E_s$ then $v$ is a minimizer of \eqref{defphi}.
\end{proposition}

\begin{proof}
 If $z$ is a calibration in the direction $p$ then as noticed above, $z$ calibrates all the solution $v_p$ of $\eqref{defphi}$.
 Let $v_p$ be one of these solutions then by the Coarea Formula and \cite[Prop. 2.7]{Anzelotti}, for every Borel set $A\subset \R^{d}$
\begin{align*}
 \int_{\R} \int_A F(x, D \chi_{E_s})  \ ds&= \int_A F(x, D v_p +p)\\
					&=\int_A [z,Dv_p +p]\\
					&=\int_{\R} \int_A [z, D \chi_{E_s}] \ ds\\
					&\le \int_{\R} \int_A F(x, D \chi_{E_s})  \ ds
\end{align*}
and thus for almost every $s \in \R$, $z$ calibrates $E_s$. Since for every $s$, $\ds E_s=\cup_{s'>s} E_{s'}$, 
by Proposition \ref{stabilitycalib}, $z$ calibrates in fact every $E_s$. 
The converse implication follows by the same argument. 
\end{proof}

\begin{theorem}\label{zfixPL}
Let $E$ be a Class A Minimizer,
let $z\in \X$ and let $\bar x$ be a Lebesgue point of $z$. Then, if $z$ calibrates $E$ and if $\bar x \in \partial E$, 
we have that $\bar x  \in \partial^*E$ and
\begin{equation}\label{zeqnu}
z(\bar x)=\nabla_p F(\bar x,\nu^E(\bar x)).
\end{equation}
\end{theorem}

\begin{proof}
Letting $z_\rho(y) = z(\bar x+\rho y)$, the assumption 
that $\bar x$ is a Lebesgue point of $z$
yields that $z_\rho\to\bar z$ in $L^1(B_R)$, hence
also weakly-$*$ in $L^\infty(B_R)$, for any $R>0$,
where $\bar z\in \R^d$ is a constant vector.

{ } As usual, we let $E_\rho := (E-\bar x)/\rho$ % $g_\rho(y)=g(\bar x + \rho y)$ (so that $\Div z_\rho=\rho g_\rho$). 
and we observe that $E_\rho$ minimizes 
%\[
%\int_{\partial E_\rho\cap B_R} F(\bar x+\rho y,\nu^{E_\rho}(y))\,d\HH(y)
%+ \rho\int_{E_\rho\cap B_R} g_\rho(y)\,dy\,,
%\]
\[
\int_{\partial E_\rho\cap B_R} F(\bar x+\rho y,\nu^{E_\rho}(y))\,d\HH(y)
\]
with respect to compactly supported
perturbations of the set (in the fixed ball $B_R$).
In particular, the sets $E_\rho$ (and the boundaries $\partial E_\rho$)
satisfy uniform density bounds, and hence are compact with respect
to both local $L^1$ and Hausdorff convergence.

{ } Hence, up to extracting a subsequence, we can assume that $E_\rho\to \bar E$, with $0\in \partial \bar E$.
Proposition~\ref{stabilitycalib} shows that $\bar z$ is a calibration
for the energy $\int_{\partial \bar E\cap B_R} F(\bar x, \nu^{\bar E}(y))\,d\HH(y)$,
and that $\bar E$ is a  plane-like minimizer calibrated by $\bar z$.

{ } It follows that $[\bar z,\nu^{\bar E}]=F(\bar x,\nu^{\bar E}(y))$ for
$\HH$-a.e.~$y$ in $\partial \bar E$, but since $\bar z$ is a constant,
we deduce that $\bar E=\{ y\cdot \bar\nu\ge 0\}$ with
$\bar \nu=F(\bar \nu)\ \nabla_p F^\circ(\bar x,\bar z)$. In particular the limit $\bar E$
is unique, hence we obtain the global convergence of $E_\rho\to \bar E$, without passing to a subsequence.

{ } We want to deduce that $\bar x\in\partial^* E$, with
$\nu^E(\bar x)=F(\bar x,\nu^{E}(\bar x))\nabla_p F^\circ(\bar x,\bar z)$, which is equivalent to \eqref{zeqnu}. 
The last identity is obvious from
the arguments above, so that we only need to show that
\begin{equation}\label{reducbound}
\lim_{\rho\to 0} \frac{ D\chi_{E_\rho}(B_1) }{ |D\chi_{E_\rho}|(B_1) }\ =\ \bar \nu\,.
\end{equation}
Assume we can show that
\begin{equation}\label{convenergy}
\lim_{\rho\to 0} |D\chi_{E_\rho}|(B_R)\ =\ |D\chi_{\bar E}|(B_R)
\ \left(\ =\ \omega_{d-1}R^{d-1}\right)
\end{equation}
for any $R>0$, then for any $\psi\in C_c^\infty(B_R;\R^d)$ we would get
\begin{multline*}
\frac{1}{|D\chi_{E_\rho}|(B_R)}\int_{B_R} \psi\cdot D\chi_{E_\rho}
\ =\ -\frac{1}{|D\chi_{E_\rho}|(B_R)}\int_{B_R\cap E_\rho} \Div \psi(x)\,dx
\\ \longrightarrow\ -\frac{1}{|D\chi_{\bar E}|(B_R)}\int_{B_R\cap \bar E} \Div \psi(x)\,dx
\ =\ \frac{1}{|D\chi_{\bar E}|(B_R)}\int_{B_R} \psi\cdot D\chi_{\bar E}
\end{multline*}
and deduce that the  
measure $D\chi_{E_\rho}/(|D\chi_{E_\rho}|(B_R))$  weakly-$*$ converges to 
$D\chi_{\bar E}/(|D\chi_{\bar E}|(B_R))$.  
Using again \eqref{convenergy}), we then obtain that 
\begin{equation}\label{aeR}
\lim_{\rho\to 0} \frac{ D\chi_{E_\rho}(B_R) }{ |D\chi_{E_\rho}|(B_R) }\ =\ \bar \nu
\end{equation}
for almost every $R>0$.
Since $D\chi_{E_\rho}(B_{\mu R})/(|D\chi_{E_\rho}|(B_{\mu R}))=D\chi_{E_{\rho/\mu}}(B_{R})/(|D\chi_{E_{\rho/\mu}}|(B_{R}))$ for any $\mu>0$, 
\eqref{aeR} holds in fact for any $R>0$ and \eqref{reducbound} follows, so that
$\bar x\in \partial^* E$.

{ } It remains to show~\eqref{convenergy}. First, we observe that,
by minimality of $E_\rho$ and $\bar E$ plus the Hausdorff convergence of $\partial E_\rho$ in balls,
we can easily show the convergence of the energies
\begin{equation}
\lim_{\rho\to 0} \int_{\partial E_\rho\cap B_R}  F(\bar x+\rho y,\nu^{E_\rho}(y))\,d\HH(y)
\ =\ \int_{\partial \bar E\cap B_R}
F(\bar x,\nu^{\bar E}(y))\,d\HH(y)
\end{equation}
%\begin{equation}
%\lim_{\rho\to 0} \int_{\partial E_\rho\cap B_R}  F(\bar x+\rho y,\nu^{E_\rho}(y))\,d\HH(y)
%+ \rho\int_{E_\rho\cap B_R} g_\rho(y)\,dy
%\ =\ \int_{\partial \bar E\cap B_R}
%F(\bar x,\nu^{\bar E}(y))\,d\HH(y)
%\end{equation}
%(the term $\rho\int_{E_\rho\cap B_R} g_\rho(y)\,dy$ clearly goes to zero).
and, by the continuity of $F$,
\begin{equation}\label{convenergyrho}
\lim_{\rho\to 0} \int_{\partial E_\rho\cap B_R}  F(\bar x,\nu^{E_\rho}(y))\,d\HH(y)
%+ \rho\int_{E_\rho\cap B_R} g_\rho(y)\,dy
\ =\ \int_{\partial \bar E\cap B_R}
F(\bar x,\nu^{\bar E}(y))\,d\HH(y)\,.
\end{equation}
Then, \eqref{reducbound} follows from a variant of Reshetnyak's continuity
theorem where instead of using the Euclidean norm as reference norm, we use the uniformly convex norm $F(\bar x, \cdot)$. Notice that this variant is covered by the original version of Reshetnyak \cite{Resh}. For the reader's convenience we sketch the proof here, following very closely the proof of \cite[Thm. 2.39]{AFP}.  In what follows, the point $\bar x$ is fixed and thus we will not specify the dependence of the functions on $\bar x$ (for example $F(p)$ will stand for $F(\bar x,p)$).

{ } Let now $\mu_\rho:=\nu^{E_\rho} \HH\LL \partial^* E_\rho$, $\mu:=\nu^{\bar E} \HH\LL \partial^* \bar E$, $\theta_\rho:=\frac{\nu^{E_\rho}}{F(\nu^{E_\rho})}$, $\theta=\frac{\nu^{\bar E}}{F(\nu^{\bar E})}$ and $W:=\{F(p)\le 1\}$. 
Then we define the measures $\eta_\rho$ on $ B_R \times \partial W$ by setting $\eta_\rho:=F(\mu_\rho)\otimes \delta_{\theta_\rho(x)}$. 
The sequence $\eta_\rho$ is bounded and thus there exists a weakly-$*$ converging subsequence to a measure $\eta$. 
Let $\pi: B_R \times \partial W \to B_R$ be the projection, then $F(\mu_\rho)=\pi_\# \eta_\rho$ and thus by \cite[Rk. 1.71]{AFP}, 
$F(\mu_\rho)$ weakly-$*$ converges to $\pi_\# \eta$, therefore by \eqref{convenergyrho} and \cite[Prop. 1.80]{AFP} we get $\pi_\#\eta=F(\mu)$.

{ } By the Disintegration Theorem \cite[Th. 2.28]{AFP}, there exists a $F(\mu)$-measurable map $x\to \eta_x$ such that $\eta_x(\partial W)=1$ and $\eta=F(\mu)\otimes \eta_x$. Arguing exactly as in \cite[Th. 2.38]{AFP}, we have
\begin{equation}\label{intpartialW}
\int_{\partial W} y d \eta_x =\theta(x) \qquad \textrm{for } F(\mu)-a.e. \ x \in B_R.
\end{equation}
The anisotropic ball $W$ being strictly convex and $\theta(x)$ being on its boundary, this will imply that indeed, $\eta_x=\delta_{\theta(x)}$ which will conclude the proof. Since $F$ is strictly convex, for every $y\in \partial W$, \eqref{strictconvF}, yields
\[
F(y)^2-F^2\left(\frac{\nu^{\bar E}}{F(\nu^{\bar E})}\right)\ge 2 F\left(\frac{\nu^{\bar E}}{F(\nu^{\bar E})}\right) 
\left( \nabla F\left(\frac{\nu^{\bar E}}{F(\nu^{\bar E})}\right)\right)\cdot \left(y-\frac{\nu^{\bar E}}{F(\nu^{\bar E})}\right) 
+ C\left|y-\frac{\nu^{\bar E}}{F(\nu^{\bar E})}\right|^2,
\]
from which it follows
\[
2\left( 1-\left( \nabla F\left(\frac{\nu^{\bar E}}{F(\nu^{\bar E})}\right)\right)\cdot y\right)\ge C\left|y-\frac{\nu^{\bar E}}{F(\nu^{\bar E})}\right|^2.
\]
Integrating this inequality on $\partial W$  and using \eqref{intpartialW} we get
\[0=2\left( 1-\left(\nabla F \left(\frac{\nu^{\bar E}}{F(\nu^{\bar E})}\right)\right)\cdot \left(\int_{\partial W} y \ d\eta_x\right)\right) \ge C\int_{\partial W}\left|y-\frac{\nu^{\bar E}}{F(\nu^{\bar E})}\right|^2\]
hence $\eta_x=\delta_{\theta(x)}$.  The proof of \eqref{convenergy} now easily follows. Indeed, since $\eta_\rho(B_R \times \partial W)$ converges to $F(\mu)(B_R)=\eta(B_R \times \partial W)$, using \cite[Prop 1.80]{AFP} we find
\begin{align*}
\lim_{\rho\to 0} \int |\theta_\rho|(x) d F(\mu_\rho)(x)&=\int_{B_R \times \partial W} |\theta_\rho(x)| \ d\eta_\rho (x,y)\\
=\int_{B_R\times \partial W} |y| d\eta(x,y) &=\int_{B_R} |\theta(x)|  d F(\mu)(x).
\end{align*}

{ } Since $|\theta_{\rho}(x)| dF(\mu_\rho)(x)=d |D\chi_{E_\rho}|(x)$ and $|\theta(x)|  d F(\mu)(x)=|D\chi_{\bar E}|(x)$, 
this gives \eqref{convenergy}.
\end{proof}
\begin{remark}\label{reman}\rm
In the isotropic case $a(x)|\nu|$, Auer and Bangert proved a similar result \cite[Th. 4.2]{auerbangert}. In that case, the monotonicity formula directly implies the convergence of the blow-up to a cone which is calibrated by the constant $\bar z$ and  is thus a plane. For minimal surfaces this classically implies that $\bar x \in \partial^* E$. 
\end{remark}
\begin{remark}
 In dimension $2$ and $3$, the converse is also true (see \cite{CGNnote}) meaning that calibrations have Lebesgue points at every regular point of a calibrated set.
\end{remark}

{ } Thanks to  Proposition \ref{maxprinc}, one can order the
% plane-like
 minimizers which are calibrated by a given vector field.

\begin{proposition}\label{order}
Let $z\in \X$ calibrates two  plane-like 
minimizers $E_1$ and $E_2$ with connected boundaries. Then, either $E_1\subset E_2$, or $E_2\subset E_1$. As a consequence $\mathcal{H}^{d-3}(\partial E_1  \cap \partial E_2)=0$.
\end{proposition} 
\begin{proof}
If $z$ calibrates $E_1$ and $E_2$ then it also calibrates $E_1\cap E_2$ and $E_1\cup E_2$ which are then also Class A
%%% NOT EVEN TRUE WITH PLANE LIKE HERE!
 minimizers by Theorem \ref{calibregalPL}. Since $E_1\cap E_2 \subset E_1$, by Proposition \ref{maxprinc}, either $E_1\cap E_2=E_1$ in which case $E_1\subset E_2$, 
or $\mathcal{H}^{d-3}(\partial (E_1\cap E_2) \cap \partial E_1)=0$ which implies that $E_2\subset E_1$.
\end{proof}

\subsection{Calibrations and the Birkhoff property}\label{birk}
We now define the class of  plane-like minimizers that we are going to consider in the analysis of the differentiability properties of $\p$. 
If $E=\{x\,:\,v_p(x)+p\cdot x>s \}$ for some $v_p$ which minimizes \eqref{defphi}, 
we have that
$E+q=\{v_p(x)+p\cdot x> s+p\cdot q\}$ for all $q\in \Z^d$, therefore $E=E+q$ if
$p\cdot q=0$, $E\subseteq E+q$ if $p\cdot q<0$, and $E+q\supseteq
E$ if $p\cdot q>0$. This is called the {\em Birkhoff property}. 
Notice that in this case we also have $E=\bigcup_{\{q\cdot p>0, \ q\in \Z^d\}} (E+q)$.  

\begin{definition}\label{defBirk}
Following \cite{HJG,S,UNIBan} we give the following definitions:
\begin{itemize}
\item we say that $E\subset\R^d$ satisfies the Birkhoff property if,
for any $q\in\Z^d$, either $E\subseteq E+q$ or $E+q\subseteq E$;
\item we say that $E$ satisfies the strong Birkhoff property
in the direction $p\in\Z^d$ if $E\subseteq E+q$ when $p\cdot q\le 0$
and $E+q\subseteq E$ when $p\cdot q\ge 0$;
\item we say that a  plane-like minimizer $E$ in the direction $p$ is recurrent if either $p$ is rational and $E$ has the strong Birkhoff property,
or if
\begin{equation}\label{PLapprox}
E=\bigcup_{q\cdot p >0, q\in \Z^d} (E+q) \quad \textrm{or} \quad E=\bigcap_{q\cdot p <0, q\in \Z^d} (E+q).
\end{equation}
\end{itemize}
\end{definition}
%Observe that if $p$ is rational, $E$ cannot be recurrent, on the other hand,
\begin{remark}\label{rkBirk}\textup{
Observe that if $E$ satisfies the Birkhoff property, there
 exists $p\in\R^d$ such that if  $q\in  \Z^d$, $q\cdot p>0$,
then $E+q\subseteq E$, while $E+q\supseteq E$ if $q\cdot p<0$
(the difference with the strong Birkhoff property is in the fact
that when $q\cdot p=0$, then one might not have $E+q=E$). The
vector $p$ (up to multiplication with a positive scalar) is uniquely
determined, unless $E+q=E$ for all $q\in\Z^d$. See~\cite{UNIBan},
or Lemma~\ref{lembirkhoffdir} in Appendix~\ref{AppBir}
for an elementary proof of this claim.
}\end{remark}
\begin{remark}\textup{A recurrent plane-like minimizer always enjoys the strong Birkhoff property (since the set $\bigcup_{q\cdot p>0} (E+q)$, for instance, obviously does).}\end{remark}

{ } We will let $\CA(p)$ be the set of all the plane-like minimizers in the direction $p$ which satisfy the strong Birkhoff property.

% It is easy to check that, if a  plane-like set $E$
% satisfies the Birkhoff property and $p$ is the direction
% in which it is  plane-like (that is, $\{ p\cdot x\ge a\}\subseteq E
% \subseteq \{p\cdot x\ge b\}$ for some real numbers $a>b$),
% then if $p\cdot q>0$ we have $E+q\subseteq E$, while if $p\cdot q<0$ we have
% $E\subseteq E+q$.  Conversely, we have
The following result can be deduced from~\cite{CDLL}.
\begin{lemma}\label{lemBKPL} If $E$ is a Class A minimizer
which satisfies the Birkhoff property, and $E \neq\emptyset$, $E\neq \R^d$,
then it is plane-like in the direction given by Remark \ref{rkBirk}. Moreover it satisfies \eqref{PLexist} with a constant $M$ depending only on the anisotropy $F$ and the dimension $d$.
\end{lemma}
\begin{proof} All the arguments can be found in the proofs of
Proposition~8.3, Proposition~8.4 and Lemma~8.5 in~\cite{CDLL}. First, if for
some $a\in \R^d$, $a+[0,1)^d\subset E$, then $E$ contains
the half-space $\{x\cdot p> a\cdot p+\sum_i |p_i|\}\subset\bigcup_{q\cdot p> 0} (q+a)+[0,1)^d$, and
similarly, if $(a+[0,1)^d)\cap E=\emptyset$, then $E$
is contained in a half-space.

Assume for instance that $E$ does not contain a half-space,
hence that $a+[0,1)^d\cap E^c\neq \emptyset$ for all $a\in \R^d$.
Then, by the density estimate, $|E^c\cap (a+[-1/2,3/2)^d)|>\delta>0$ 
for any $a$ and some constant $\delta$ which depends only on $c_0$ and the dimension $d$. 
Now, we also have that $b+[0,1)^d\cap E\neq \emptyset$ for some
$b\in\R^d$, otherwise $E$ would be empty. Then, for any $q\in \Z^d$ with $q\cdot p\ge 0$,
$(q+b+[0,1)^d)\cap E\subset (q+b+[0,1)^d)\cap (q+E) \neq \emptyset$.
Again it follows that $|(q+b+[-1/2,3/2)^d)\cap E|\ge \delta>0$.
We deduce that the energy $\int_{q+b+[-1/2,3/2)^d} F(x,D\chi_E)$ is bounded
from below, by some constant $\delta'>0$.
Hence, if $B_R(x_R)$ is a large ball contained in $\{ x\, : \, (x-b)\cdot p\ge 0\}$,
the energy in the ball is bounded below by $N \delta'$, 
where $N:=\# \{q\in \Z^d \, : \,  q+b+[-1/2,3/2)^d\subset B_R(x_R)\}\cong R^d$.
However, by Class A
minimality it is also less than $c_0^{-1}d\omega_d R^{d-1}$, a contradiction.
It follows that $E$ satisfies \eqref{PLexist}, with a constant $M$ independent on $E$.
\end{proof}
% If $p\cdot q=0$, one might have either
%or both inclusions (in particular,
%in case of heteroclinic surfaces, see Section~\ref{sechetero},
%the equality does not always hold).

%{ } The following result is true:
\begin{proposition}\label{propcalib}
Let $E$ be a Class~A Minimizer with the Birkhoff property,
%set of finite perimeter, with $E\neq \emptyset$, $E\neq \R^d$.
%Then the following are equivalent:
%\begin{itemize}
%\item $E$ is a ;
%%then there exists a periodic calibration of $E$;
%\item  
then $E$ has a periodic calibration.
%%, then it is a Class-A Minimizer with the Birkhoff property.
%\end{itemize}
\end{proposition}

\begin{proof}
%Consider a Class~A Minimizer $E$, which
%% is  plane-like in
%satisfies the Birkhoff property. %% in the direction $p$.
%% From Lemma~\ref{lemBKPL}, $E$ we observe that $E$ is plane-like.
%Let us show that there exists
%a calibration . 
Let $R>0$ and $k\ge 1$, and let 
\[
v_k(x)\ :=\ \sum_{q\in\Z^d, |q|\le k} \chi_{E+q} \ \in\ BV(B_R)
\]
where in the sum, we drop the terms which are $1$ a.e. on $B_R$.
%% (since $E$ is  plane-like, the sum is then finite).
Thanks to the Birkhoff property, the sets $E+q$, $|q|\le k$ are
exactly the level sets of $v_k$. Consider now $v\in BV(B_R)$
such that $v-v_k$ has support in $B_R$.
For $s\in \R$, the level set $\{v>s\}$ is a compactly 
supported perturbation  of the level set $\{v_k>s\}$. Since this
latter set is a Class~A Minimizer, one has
\begin{equation}
\int_{B_R} F(x,D\chi_{\{ v_k>s\}})\ \le\ \int_{B_R} F(x,D\chi_{\{v>s\}})\,.
\end{equation}
Hence,
\begin{multline}
\int_{B_R} F(x,Dv)\, =\, \int_{-\infty}^\infty \int_{B_R} F(x,D\chi_{\{v>s\}})\,ds
\\ \ge\, \int_{-\infty}^\infty \int_{B_R} F(x,D\chi_{\{v_k>s\}})\,ds
\, =\, \int_{B_R} F(x,D v_k)\,,
\end{multline}
%\begin{equation}
%\int_{B_R} F(x,Dv) \,=\, \int_{-\infty}^\infty \int_{B_R} F(x,D\chi_{\{v>s\}})\%,ds
%\, \ge\, \int_{-\infty}^\infty \int_{B_R} F(x,D\chi_{\{v_k>s\}})\,ds
%\,=\, \int_{B_R} F(x,D v_k)\,,
%\end{equation}
so that $v_k$ is minimizing in $B_R$, with its own boundary datum.
This yields the existence of a calibration $z_k^R\in L^\infty(B_R;\R^d)$,
such that $F^\circ(x,z_k^R(x))\le 1$ a.e., $\Div z_k^R=0$, and
$[z_k^R, Dv_k]=F(x,Dv_k)$ (in the sense of measures).
By construction, the latter property is equivalent to
$[z_k^R, D\chi_{E+q}]=F(x,D\chi_{E+q})$ for any $q\in \Z^d$ with $|q|\le k$,
that is, $z_k^R$ is also a calibration for each minimizing set $E+q$,
inside $B_R$.

{ } Now, we let $k\to\infty$:  up to a subsequence, $z_k^R$ will
converges, weakly-$*$ in $L^\infty(B_R;\R^d)$,
to some $z_R$ which will be a calibration for all the sets $E+q$, $q\in\Z^d$,
in the ball $B_R$ (\textit{cf} Prop.~\ref{stabilitycalib}).
Then we can send $R\to\infty$, in that case $z_R$ (extended
by zero out of $B_R$) converges again, weakly-$*$ in $L^\infty(\R^d;\R^d)$,
to some $z$ which is a calibration for all the sets $E+q$, $q\in\Z^d$,
in any ball.
Let us now show that $z$ may be chosen to be periodic: indeed, clearly,
$z(x-q)$ is also a calibration for $E$ and all its translates, for
any $q\in\Z^d$. One may consider for any $k$
\begin{equation}
z'_k(x)\ 
:=\ \frac{1}{\#\{q\in\Z^d\,:\,|q|\le k\}}\sum_{q\in\Z^d\,, |q|\le k} z(x-q)\,.
\end{equation}
which again, will be a calibration for $E$ and all its translates.
Passing to the limit, it converges (up to subsequences) to a new
calibration $z'$, which is now periodic.
\end{proof}

\begin{proposition}\label{proppo}
Let $E$ be a plane-like minimizer with a periodic calibration, then
$E$ has the Birkhoff property and $\partial E$ is connected.
\end{proposition}

\begin{proof}
%The assertion follows from Theorem \ref{calibregalPL},
%Proposition \ref{order} and the fact that, if $z$ is a periodic
%calibration of $E$, then it also calibrates all the integer translates of $E$. 
%
%$E$ has a periodic calibration which then calibrates every connected component of $E$. Again by Proposition \ref{propcalib}, these connected components are thus also Class A Minimizers with the Birkhoff property. Lemma \ref{lemBKPL} implies that every connected component is plane-like, that is, $E$ is connected. The same argument gives that $E^c$ is also connected.
Without loss of generality we can assume that $E$ satisfies \eqref{PLexist} with $a=0$.
Since $E$ has a periodic calibration, by Theorem \ref{calibregalPL} it is a Class A Minimizer. Moreover, since every connected component 
$E$ is also calibrated, any of them is a Class A Minimizer. 
Let $E_0$ be the connected component of $E$ which contains the half space $\{x\cdot p>M\}$ and let $E_1$ be another connected component of $E$.
By Proposition \ref{propdensity}, for every $R>0$
we have $\E(E_1,B_{R})\ge c_0\beta R^{d-1}$. 
However, since $E_1\subset \{ |x\cdot p|\le M\}$, the minimality of
$E_1$ also yields $\E( E_1, B_{R})\le CM R^{d-2}$. This is a contradiction
if $R$ is large enough, and it follows that $E$ is connected.

An analogous argument gives that $E^c$ is also connected, thus implying the thesis. The Birkhoff property is deduced from Proposition~\ref{order}, applied
to $E$ and $q+E$, $q\in\Z^d$.
\end{proof}

From Lemma \ref{lemBKPL} and Propositions \ref{propcalib} and~\ref{proppo}
we obtain the following
\begin{corollary}\label{coroconnect}
Let $E$ be a Class A Minimizer with the Birkhoff property, then $\partial E$ is connected. 
\end{corollary}

\begin{remark}\rm
An interesting question raised by Bangert in \cite{Bangertmin} for non-parametric integrands is whether every  plane-like minimizer necessarily satisfies the Birkhoff property. In \cite[Th. 8.4]{Bangertmin}, Bangert proves that, in the non-parametric case, it is true for totally irrational vectors $p$. 
Propositions \ref{propcalib} and \ref{proppo} show that, in the parametric case, this question  is equivalent to understand if every  plane-like minimizer has a periodic calibration. See also \cite{JGV} where a nice relation is given between this question of Bangert and De Giorgi's conjecture. 
\end{remark}

{ } We also show the following result:
\begin{proposition} \label{Proprecurrent}
$E$ is a recurrent plane-like minimizer in the direction $p$ 
%% with the strong Birkhoff property
if and only if there exists a minimizer $v_p$ of~\eqref{defphi} such that
$$E=\{x\,:\,v_p(x)+p\cdot x>0\} \quad
or \quad E=\{x\,:\,v_p(x)+p\cdot x\ge 0\}.$$
\end{proposition}
\begin{proof}
The ``if part'' is straightforward, as already observed. \\
If $E$ is a recurrent  plane-like minimizer (hence with the
strong Birkhoff property), by
Proposition~\ref{propcalib} $E$ has a periodic calibration $z$.
We build a function $v_k$, $k\ge 1$, as follows:
since $E$ has the strong Birkhoff property, we can define in $BV_\loc(\R^d)$
a function $v_k$ such that $\{v_k \ge p\cdot q\}= E+q$ for
all $q$ with $|q|\le k$.
Indeed, $E$ being  plane-like, if $p\cdot q>0$ one cannot
have $E+q=E$, otherwise repeating the translation one would reach
a contradiction. Hence the function
\begin{equation}
v_k(x) \ :=\ \sup\{ p\cdot q\,:\, |q|\le k, x\in E+q\}
\end{equation}
has actually the right level sets. Now, since $E$ is  plane-like,
its oscillation is also uniformly bounded, and in fact
it is locally uniformly bounded in $L^\infty$.
By construction, $z$ is a calibration for $v_k$, which means in particular
that for any $R>0$,
\[
\int_{B_R} F(x,Dv_k) \ =\ \int_{B_R} [z, Dv_k]
\ =\ -\int_{\partial B_R} v_k [z,\nu^{B_R}] \ \le\ 2^d C_R c_0^{-1}
\]
where $C_R$ is a uniform bound for $\|v_k\|_{L^\infty(B_R)}$.
Hence the $v_k$ are uniformly bounded in $BV(B_R)$: up
to a subsequence, we may assume that $v_k\to v$ in $L^1_\loc(\R^d)$
with $v\in BV_\loc(\R^d)$. 

{ } If $|q|\le k$, then $E+q\subseteq \{v_k\ge p\cdot q\}$. Passing
to the limit it follows that $E+q\subseteq \{v\ge p\cdot q\}$.
Conversely, if $(E+q)^c\subseteq \{v_k< p\cdot q\}$. Hence
$(E+q)^c\subseteq \{v_k\le p\cdot q\}$. Then, for any $q\in\Z^d$,
\begin{equation}\label{inclusionEq}
 \{v > p\cdot q\}\subseteq E+q\subseteq \{v\ge p\cdot q\}.
\end{equation}
If $p$ is rational, then it is obvious that equality holds in
 \eqref{inclusionEq}, since, in fact, one can check that $v_k$ does not change
when $k$ is large enough.
If $p$ is not rational, since $E$ is recurrent we can assume that $E=\bigcup_{p\cdot q>0} (E+q)$, and we shall prove that $E\subset \{v>0\}$.
This will imply that $E=\{v>0\}$, and $E+q=\{v>p\cdot q\}$ for every $q\in \Z^d$.

{ } If $x\in E$ then, for some $q\in \Z^d$ with $p\cdot q >0$, we have $x\in E+q$ and thus, for $k\ge|q|$, we also have $v_k(x)\ge p\cdot q>0$. 
Since the sequence $v_k(x)$ is increasing, we get $v(x)>0$ and thus $E\subset \{v>0\}$. The case $E=\bigcap_{p\cdot q<0} (E+q)$ is similar and gives $E=\{v\ge 0\}$.

{ } Let us show that $v-p\cdot x$ is periodic. It is enough to
show that for almost every $t$ and for all $q\in\Z^d$, we have
$\{v\ge t\}+q\subset \{x\,:\, v\ge t+ p\cdot q\}$.
Then, letting $v_p(x)=v(x)-p\cdot x$, we will deduce that
$$v_p(x) \ge t-p\cdot x\ \Longrightarrow\ v_p(x+q)\ge t+ p\cdot q-p\cdot (x+q)
=t-p\cdot x$$ 
for almost every $t,x$ and for all $q\in\Z^d$,
yielding that $v_p$ is periodic (indeed, being $\Z^d$ countable,
the converse also holds for a.e.~$t$ and $x$).

{ } For a.e.~$t$ we have $E_t=\{v\ge t\}=\{v>t\}$. In that case, $E_t$ is
the Kuratowski limit of $E+q_n$ for some sequence $(q_n)$ in $\Z^d$, with $p\cdot q_n\to t$ as $n\to +\infty$.
In particular, $E_t$ is calibrated by $z$.
Now, for $k$ large enough and fixed $n$,
$E+q_n+q=\{v_k\ge p\cdot (q_n+q)\}$, hence in the limit $E+q_n+q
\subseteq \{v\ge t+p\cdot q\}$. Passing then to the limit in $n$
we find that $\{v\ge t\}+q\subseteq \{v\ge t+p\cdot q\}$, which
is our claim.

{ } As already observed, almost all level sets of $v$ are
calibrated by $z$. It follows easily that $v_p$ is a minimizer of~\eqref{defphi}. 
\end{proof}

{ } An example of a non recurrent plane-like minimizer satisfying the strong Birkhoff property would be a set $E$ with gaps on boths side of its
boundary, however we do not know whether this situation can occur.

% is given by a   plane-like minimizer $E$ which is the boundary of two different gaps $G_1$ and $G_2$ (see Figure \ref{nonrecurrent}). 
% It is  however not clear if this situation could occur or not. 
% \begin{figure}[ht]
% \centering{\input{figuresPL/nonrecurrent.pstex_t}}
% \caption{Example of non recurrent plane-like minimizer.}
% \label{nonrecurrent}
% \end{figure}

{ } Thanks to Proposition \ref{Proprecurrent} we can now prove that every calibration calibrates every plane-like minimizer with the strong Birkhoff property, which implies that $\CA(p)$ is stable under union or intersection.

\begin{theorem}\label{calibrallPL}
Let  $z$ be a calibration in the direction $p$,
then $z$ calibrates every plane-like minimizer
with the strong Birkhoff property.
\end{theorem}

\begin{proof}
By Proposition \ref{Proprecurrent}
%Let $p\in \R^d$ be fixed then if $p$ is rational 
we know that every recurrent plane-like minimizer
% satisfying the strong Birkhoff property
is of the form $\{v_p + p\cdot x > s\}$ or $\{v_p + p\cdot x \ge s\}$, 
for some periodic function $v_p$ minimizing \eqref{defphi}, which implies that it is calibrated by every calibration.

{ } We can thus assume that $p$ is irrational and that we are given a non recurrent plane-like minimizer $E$, but which satisfies
the strong Birkhoff property. 
%If $E$ is recurrent then by Proposition \ref{Proprecurrent}, as in the periodic case, $E$ is calibrated by any calibration. 
Then, the set $\widetilde E:=\bigcup_{q\cdot p >0} (E+q)$
is a plane-like recurrent Class A Minimizer (Proposition~\ref{stabPL}),
hence calibrated by $z$.
%set with the strong Birkhoff property 
Moreover, it satisfies 
\begin{equation}\label{tildeE}
\widetilde E \subset E \qquad \textrm{ and } \qquad E\subset \widetilde E +q \quad  \forall q\in \Z^d \,,\quad  %\textrm{s.t. }
 q\cdot p <0. 
\end{equation}

{ } Using the notation of Section \ref{notation}, we define  the quotient torus with respect to the rational directions orthogonal to $p$, 
$\T_r:= \R^d / \Gamma(p)$.
%%, where the equivalence relation $\sim$ is given by:
%%%
%%\[ 
%%x \sim y\quad  \Leftrightarrow \quad y-x\ \in\ \Gamma(p)\,.
%%\]

{ } If we are given a calibration $z$ in the direction $p$, since it is periodic and since $E$ and $\widetilde E$ are invariant by translations in $\Gamma(p)$, we can identify them with their equivalence class in $\T_r$. Thanks to \eqref{tildeE}, the measure $|E \setminus \widetilde E|$ is then finite.
Reminding that $V^r(p)=\Span_\R \Gamma(p)$, we let for every
$x\in\R^d$, $f(x):=\min_{x^r\in V^r(p)} |x-x^r|$. Then, the projection
of $f$ on $\T_r$ (still denoted by $f$)
is well defined, and satisfies $|\nabla f|=1$ a.e.~in $\T^r$.
Given $s\in \R$ we let $C_s:=\{x\in\T_r\,:\, f(x)\le s\}$. Then,
it follows (from the Coarea Formula)
that
\[
|E\setminus\widetilde{E}|
\ =\ \int_0^\infty \HH((E\setminus \widetilde E )\cap \partial C_s)\,ds  \,.
\]
Hence, there exists a sequence $s_i\to +\infty$ such that
\begin{equation}\label{estimdiff}
   \HH((E\setminus \widetilde E )\cap \partial C_{s_i}) \le \frac{1}{s_i}.
\end{equation}

Now, Proposition~\ref{propcalib} shows that every plane-like
minimizer which satisfies the Birkhoff property is calibrated
by a periodic calibration, and we can easily deduce that its projection
in $\T_r$
is also minimizing (with respect to compact perturbation inside the cylinder).
In particular, comparing the energy of $E$ with the energy
of $(E\setminus  C_{s_i}) \cup (\widetilde E \cap C_{s_i})$, it follows
from~\eqref{estimdiff}
\begin{equation}\label{AAA1}
\int_{\partial^* E \cap C_{s_i}} F(x,\nu^E) \le \int_{\partial^* \widetilde E \cap C_{s_i}} F(x,\nu^{\widetilde E}) + \frac{c_0^{-1}}{s_i}\,.
\end{equation}
Integrating by parts on $(E\setminus \widetilde E)\cap C_{s_i}$ and recalling that $z$ calibrates $\widetilde E$, we get (using~\eqref{estimdiff} once more)
%\begin{align*}
\begin{equation}\label{AAA2}
\int_{\partial^* E \cap C_{s_i}} [z,\nu^E] \ 
=\ \int_{\partial^* \widetilde E \cap C_{s_i}} [z,\nu^{\widetilde E}]
- \int_{\partial C_{s_i} \cap (E\setminus \widetilde E)} [z,\nu^{C_{s_i}}]
\ \ge\ \int_{\partial^* \widetilde E \cap C_{s_i}} F(x,\nu^{\widetilde E}) - \frac{c_0^{-1}}{s_i}\,.
\end{equation}
%% \\&\ge \int_{\partial^* E \cap C_{s_i}} F(x,\nu^E) - \frac{2c_0^{-1}}{s_i}.
%\end{align*} 
Estimates~\eqref{AAA1}, \eqref{AAA2} yield
\[
\int_{\partial^* E \cap C_{s_i}} F(x,\nu^E) \ \le\ 
\int_{\partial^* E \cap C_{s_i}} [z,\nu^E]\,+\,2\frac{c_0^{-1}}{s_i}
\]
and since $[z,\nu^E]\le F(x,\nu^E)$ a.e.~on $\partial^* E$,
we easily deduce that in any $C_R$, we must have
$\int_{\partial^* E \cap C_R} [z,\nu^E]=\int_{\partial^* E \cap C_R} F(x,\nu^E)$
so that $z$ calibrates $E$.
\end{proof}

\begin{corollary}\label{corcomparison}
For every $p\in \R^d\setminus\{0\}$, given $E$, $F$ two plane-like
minimizers in
the direction $p$, if $E$ has the Birkhoff property, and $F$ the
strong Birkhoff property, then either $E\subset F$ or $F\subset E$.
\end{corollary}
\begin{proof}
If follows from Proposition~\ref{propcalib}, Corollary~\ref{coroconnect},
Proposition~\ref{order} and Theorem~\ref{calibrallPL}.
\end{proof}
\begin{corollary}
For every $p\in \R^d\setminus\{0\}$, the   plane-like minimizers of $\CA(p)$ form a lamination of $\R^d$ (possibly with gaps).
\end{corollary}

The following Lemma is reminiscent of \cite[Lemma 4.4]{UNIBan} and shows that for irrational vectors $p$, every recurrent plane-like minimizer is in the orbit of every other minimizer with the Birkhoff property.
\begin{lemma}\label{lemrecorbit}
Let $p\in \R^d \setminus (\R \cdot \Z^d)$ be an  irrational vector. Then for every  recurrent plane-like minimizer $E$ satisfying
\begin{equation}\label{recapproxbelow} 
E=\bigcup_{q\in \Z^d\!,\,q\cdot p >0} (E+q)
\end{equation}
and for every plane-like minimizer  $F$ with the Birkhoff property, there holds
\begin{equation}\label{recegal}
E =\bigcup_{q\in \Z^d\!,\, q\cdot p>\alpha } (F+q)
\end{equation} 
where $$\alpha:=\inf \{ q\cdot p \ : \ q\in \Z^d \textrm{ and }  \ F+q \subset E\}.$$ 
In particular, for every $u_p$ and $v_p$ minimizing \eqref{defphi} and every $s, t\in \R$, there exists $\alpha \in \R$ such that $\{u_p+p\cdot x >t\}=\{ v_p+p\cdot x> s+\alpha\}$. A similar result holds for the recurrent sets satisfying the second equality in \eqref{PLapprox}.
\end{lemma}

\begin{proof}
Let $E,\,F$ as above.
%For $E$ a recurrent minimizer, $F$ a plane-like minimizer satisfying the Birkhoff property and , 
By Corollary \ref{corcomparison}, for all $q\in \Z^d$,
either $E\subset F+q$ or $F+q\subset E$. 
Notice that $\alpha$ is well-defined
since $E$ and $F$ are plane-like (in the direction $p$). Let 
\begin{equation}\label{eqinc}
\widetilde F:=\bigcup_{ q\in \Z^d \!,\, q\cdot p>\alpha} (F+q)\ \subset\ E
\end{equation} 
and assume that the inclusion in \eqref{eqinc} is strict.
Thanks to \eqref{recapproxbelow}, there exists $\bar q\in \Z^d$ with $\bar q\cdot p>0$ and such that for every $q\in \Z^d$ with $p\cdot q>\alpha$ there holds 
 \[F+q\ \subset\ \widetilde F\ \subset\ E+\bar q \ \subset\ E\]
 and thus 
 \[F+ (q-\bar q)\ \subset\ E\]
 which contradicts the definition of $\alpha$. 
 
Applying this to the recurrent sets $\{u_p+p\cdot x >t \}$ and $\{ v_p+p\cdot x> t+\alpha\}$ which both satisfy \eqref{recapproxbelow} 
and recalling that $\{ v_p+p\cdot x> s+\alpha\}=\bigcup_{q\in \Z^d\!,\,p\cdot q >\alpha} \{ v_p+p\cdot x> s\} +q$, we conclude the proof of the lemma.
\end{proof}

As a consequence we get the following uniqueness result.
%a weak version of  Ma\~n\'e's conjecture \cite{mane}.??
\begin{theorem}\label{thmane}
Let $p\in \R^d \setminus (\R \cdot \Z^d)$ be an irrational vector, then the minimizer $v_p$ of \eqref{defphi} is unique up to an additive constant.
\end{theorem}
\begin{proof}
Let $u_p$ and $v_p$ be two minimizers of \eqref{defphi}. By Lemma \ref{lemrecorbit} 
there exists $\alpha\in \R$ such that $\{v_p+ p\cdot x >0\}= \{u_p+ p\cdot x>\alpha\}$. Then, for any  $t\in \R$ there holds
% \begin{align*}
% \{v_p+ p\cdot x >t\}\ =&\bigcup_{q\in \Z^d\!,\,q\cdot p> t} \{v_p+ p\cdot x >0\}+ q
% \\
% \ = &\bigcup_{q\in \Z^d \!,\, q\cdot p> t} \{u_p+ p\cdot x >\alpha\}+ q
% \\
% \ = &\bigcup_{q\in \Z^d \!,\, q\cdot p> t+\alpha} \{u_p+ p\cdot x >0\}+q\\
% \ =\ &\{u_p+ p\cdot x >t+\alpha\}.\end{align*}
\begin{multline*}
\{v_p+ p\cdot x >t\}\ =
\bigcup_{q\in \Z^d \!,\, q\cdot p> t} \{v_p+ p\cdot x >0\}\,+\, q
 = \bigcup_{q\in \Z^d \!,\, q\cdot p> t} \{u_p+ p\cdot x >\alpha\}\,+\, q
\\
\ = \bigcup_{q\in \Z^d \!,\, q\cdot p> t+\alpha} \{u_p+ p\cdot x >0\}\,+\,q
\ =\ \{u_p+ p\cdot x >t+\alpha\}\,.
\end{multline*}
%% Hence, by the Coarea Formula, $D (v_p + p\cdot x)= D(u_p+ p\cdot x)$ and thus
It follows that $u_p=v_p+\alpha$.
\end{proof}

\begin{remark} \rm
In our context, the measure $D v_p + p$ plays the role of Mather's measures in Weak KAM Theory, and of the minimizing currents in the non-parametric setting. In that context, Bessi and Massart  \cite{bessiMassart} proved that for irrational directions every non self-intersecting minimizer gives rise to the same minimizing current. In some sense their result is stronger than ours since the measure $D v_p +p$ only accounts for the recurrent minimizers. 
In the same paper they also prove Ma\~n\'e's conjecture \cite{mane}, namely
that the uniqueness result generically holds also in the rational directions. 
See Appendix~\ref{SecGeneric} for a similar result in our context.
\end{remark}

{ } From Theorem \ref{calibrallPL}, Theorem \ref{zfixPL} and Proposition \ref{supDu} we have

\begin{theorem}\label{zfix}
Let $p\in \R^d\setminus\{0\}$ and $\Lambda:=\bigcup \{\partial^*E \,:\, E\in \CA(p)\}$ then for every calibration $z\in X$ in the direction $p$ we have
\[z=\nabla_p F(x,\nu^E(x)) \qquad a.e. \textrm{ on } \Lambda,\]
where $\nu^E(x)$ is the normal to the   plane-like minimizer passing through $x$. 
In particular, if $v_p$ is a minimizer of \eqref{defphi}, then $z$ is prescribed almost everywhere in $\textrm{Spt}(|Dv_p+p|)$. 
\end{theorem}

\subsection{Heteroclinic surfaces}\label{sechetero}
We consider now $p\in\R^d\setminus\{0\}$ a non totally irrational vector and let us assume that there exists $x\in \R^d$ such that no 
plane-like minimizer in the direction $p$, with the strong Birkhoff property,
passes through $x$. Let
\begin{equation}
E^+ \ =\ \bigcap \left\{ E\,:\, E\in \CA(p)\,, x\in E\right\}
\,,\qquad 
E^-\ =\ \bigcup \left\{ E\,:\, E\in \CA(p)\,, x\not\in E\right\}\,.
\end{equation}
Then, there exists an open set $G=\Int(E^+)\setminus \overline{E^-}$, called a gap,
which contains $x$, and through which no plane-like minimizer with
the strong Birkhoff property passes.

\begin{figure}[ht]
\centering{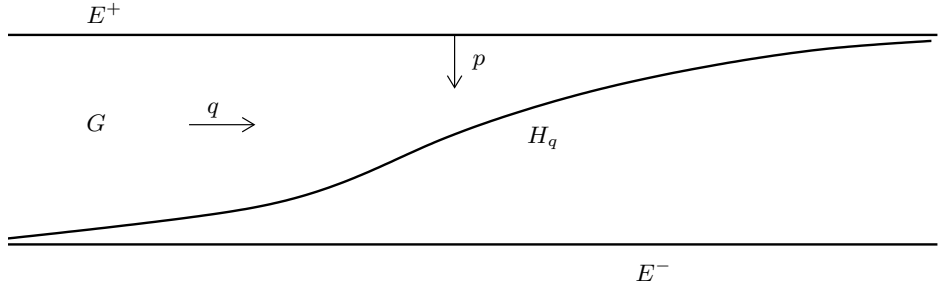}
\caption{Heteroclinic  surface inside a gap.}
\label{hetero}
\end{figure}

{ } We show that for any $q\in V^r(p)$, one can build inside $G$
a heteroclinic surface which is ``growing in the direction $q$'',
see Fig.~\ref{hetero}.
Let us remind that $\Gamma(p)=V^r(p)\cap \Z^d=\{\alpha\in \Z^d \,:\,  \alpha\cdot p=0\}$.
\begin{definition}
Let $p\in \R^d\setminus\{0\}$ be a non totally irrational vector and let $q\in V^r(p)$. If $H$ is a   plane-like minimizer in the direction $p$ satisfying the Birkhoff property we will say that $H$ is a heterocline in the direction $q$ if 
\begin{itemize}
 \item for every $\alpha \in \Gamma(p)$ with $\alpha\cdot q=0$, we have $H+\alpha=H$;
\item for every $\alpha \in \Gamma(p)$ with  $\alpha\cdot q>0$, there holds $H+\alpha\subset H$.
\end{itemize}
\end{definition}

\begin{proposition}\label{prophetero}
Let $p\in \R^d\setminus\{0\}$ be a  non  totally irrational vector, and assume that there exist $E^\pm$, $G$ as above. 
Then, there exists a compact set $K\subset G$ such that,  for every $q\in V^r(p)$, there exists a heterocline $H_q$ in the direction $q$ whose boundary intersects $K$. 
Moreover, letting 
\begin{equation}\label{eteroclinix}
H^+_q:= \bigcup_{\alpha\in \Gamma(p)} \left( H_q+ \alpha\right) \qquad H^-_q:=\bigcap_{\alpha\in \Gamma(p)} \left( H_q+ \alpha\right)
\end{equation}
we have 
\begin{equation}\label{eteroclinix2}
H^+_q=E^+ \quad and \quad H_q^-=E^-.
\end{equation}
\end{proposition}

\begin{proof}

Let $(\bar q_1, \dots,\bar q_k)\in \Z^d$ be an orthogonal basis of $V^r(p)$.
We choose a (very) small $\eta>0$ such that $\{x\in G \, : \, \dist(x,\partial G)\ge \eta\}\neq \emptyset$, and let
\[
K:= \bigcap_{i=1}^k \left\{0\le x\cdot (\bar q_i/|\bar q_i|)\le |\bar q_i|\right\}\cap \{x\in G \, : \, \dist(x,\partial G)\ge \eta\}.
\]
Following~\cite{Bangertmin,HJG,S}, the idea consists in letting $p_n=p+\frac{1}{n}q$ and considering a   plane-like minimizer $E_{p_n}\in \CA(p_n)$ calibrated by a vector field $z_n$. Since $E_{p_n}$ is included in a strip with normal $p_n$,  up to an integer translation, we can assume that $\partial^* E_{p_n}\cap K \neq \emptyset$ for every $n$. Thus there exists a subsequence of $E_{p_n}$ converging to a plane-like set $H_q$ which is calibrated by $z:= \lim z_n$,
and such that $\partial H_q\cap K\neq \emptyset$ (by Hausdorff convergence
in $K$ of the boundaries).
Moreover, $z$ is a calibration in the direction $p$,
since by weak-$*$ convergence we also have
\[
\left(\int_Q z \right) \cdot p=\p(p)\,.
\]
By construction $H_q$ satisfies the Birkhoff property, is periodic in all the rational directions which are orthogonal to $p$ and $q$, and 
$H+\alpha\subset H$ for all $\alpha \in \Gamma(p)$ with $\alpha\cdot q>0$.

Define now $H^\pm_q$ as in~\eqref{eteroclinix}.
It is clear that $H^\pm_q \in \CA(p)$ and thus, by definition of $G$, 
$H_q^- \subset E^- \subset E^+ \subset H_q^+$. If $z$ is the  calibration of $H_q$ given above, $z$ calibrates also $H_q^\pm$ and $E^\pm$  and thus,  by the maximum principle, $H_q \subset E^+$.
Therefore,  

$H_q+\alpha \subset E^+$ for every $\alpha\in \Gamma(p)$, from which we get $H_q^+ \subset E^+$. Similarly we have $E^- \subset H^-_q$, so that~\eqref{eteroclinix2} is true.
\end{proof}

\begin{remark}\rm
The previous proposition asserts that the set $H_q$ is really a heteroclinic solution, in the sense that it is a   plane-like minimizer in the direction $p$ which connects two periodic plane-like minimizers in the same direction.
\end{remark}

\begin{remark}\rm
{ } When investigating the differentiability properties of $\p$ at a point $p$ in the direction $q$, it is natural to consider the heteroclinic minimizers constructed as above, which amounts to studying the asymptotic behavior of $\p(p+\frac{1}{n} q)$ as $n\to +\infty$.
\end{remark}
 
We can prove the following uniform energy estimate for the heteroclinic surfaces:

\begin{proposition}\label{estimdelta}
There exists $\delta>0$ and $R>0$ such that, for every $q\in V^r(p)$ and every $H_q, H_{-q}$ heteroclinic surfaces in the direction $q, -q$ respectively, such that $\partial^*H_q\cap \partial^*H_{-q}\cap K \neq \emptyset$, where $K$ is the compact set given in Proposition \ref{prophetero}, there holds 
\[
\int_{B_R\cap \partial^* H_q} F(x, \nu^{H_q})-[z,\nu^{H_q}]\ \ge\ \delta
\]
for every $z\in X$ calibrating $H_{-q}$.
\end{proposition}

\begin{proof}
Assume on the contrary that there exists $R_n\to +\infty$ and $\delta_n\to 0$  such that there exists $q_n\in V^r(p)$, with $|q_n|=1$, $H_{q_n}$ heteroclinic in the direction $q_n$, 
and $H_{-q_n}$ heteroclinic in the direction $-q_n$, calibrated by $z_n\in X$ and such that
\begin{equation}\label{quasicalib}
\int_{ B_{R_n}\cap \partial^* H_{q_n}} F(x, \nu^{H_{q_n}})-[z_n,\nu^{H_{q_n}}] \le \delta_n.
\end{equation}
Since for every $n$, $\partial^* H_{\pm q_n}$ intersects $K$, there exists a subsequence such that $q_n\to q$, $H_{\pm q_n} \to H_{\pm}$, with $\partial H_+\cap \partial H_{-}\cap K \neq \emptyset$ and $z_n$ converge weakly-$*$ to $z\in X$, where $z$ calibrates $H_-$.
By Proposition \ref{stabPL}, Proposition \ref{stabilitycalib} and Proposition \ref{propcalib}, $H_{\pm}$ are plane-like minimizers with the Birkhoff property. Even if it is not clear that $H_{\pm}$ are heteroclinic, 
for any $\alpha \in\Gamma(p)$ one of 
the following properties hold
 \begin{enumerate}
  \item  $H_+ +\alpha\subset H_+ $ and $H_- \subset H_- +\alpha$
  \item  $H_+  \subset H_++\alpha$ and $H_- +\alpha \subset H_- $.
\end{enumerate} 

Arguing as in the proof of Proposition \ref{stabilitycalib}, we
easily deduce from~\eqref{quasicalib} that $z$ is a calibration also for $H_+$.
% For every $R>0$ we have,
% \[\varliminf_{n\to +\infty} \int_{B_R} F(x, D\chi^{}_{H_{q_n}})\ge \int_{B_R} F(x, D\chi^{}_{H_{+}})\]
% and by the Hausdorff convergence of $\partial H_{q_n}$ to $\partial H_+$ and the weak-$*$ convergence of $z_n$, 
% arguing as in Proposition \ref{stabilitycalib}  we have 
% \[
% \lim_{n\to +\infty} \int_{B_R \cap \partial^* H_{q_n}} [z_n, \nu^{H_{q_n}}]=\int_{B_R \cap \partial^* H_{+}} [z, \nu^{H_{+}}]
% \]
% and thus 
% \[
% \int_{B_R \cap \partial^* H_{+}} [z, \nu^{H_{+}}]=\int_{B_R} F(x, D\chi?^{}_{H_{+}}),
% \]
% so that $z$ also calibrates $H_+$.
By Proposition \ref{order} we deduce that
\begin{equation}\label{eqimp}
H_-\subset H_+ \qquad \textrm{or} \qquad H_+\subset H_-.
\end{equation}

%%% non capisco questa costruzione, ne' come si puo "repeating the construction"
%%% quando non c'e` piu`, pare, l'inclusione $H^\alpha_-\subset H_+^\alpha$
%%% ma ho l'impressione che si potrebbe scrivere la stessa cosa a modo che
%%% non sia utile. propongo una variante...

%  Let also $\alpha\in \Gamma(p)$ be such that 1.~holds: $H_++\alpha\subset H_+$ and $H_-\subset H_-+\alpha$. 
% Then 
% \[
% x\in \overline{H}_-\subset \alpha+\overline{H}_-\subset \alpha + \overline{H}_+ \qquad \textrm{and} \qquad  x\in H_+^c\subset \alpha+ H_+^c\subset \alpha + H_-^c
% \]
% so that, letting 
% \[
% H_-^\alpha:= \bigcup_{n\ge 0} (n\alpha + H_-) \qquad \textrm{and } \qquad H_+^\alpha:=\bigcap_{n\ge 0} (n\alpha + H_+),
% \]

% there holds $H_{\pm}^\alpha+\alpha =H_{\pm}^\alpha$ and $x\in \overline{H^\alpha}_-\cap (H_+^\alpha)^c$. Repeating the same construction for every element of a basis of $\Gamma(p)$ over $\Z$, 
% we find two plane-like minimizers $\widetilde H_{\pm}$ such that 
% $\widetilde H_{\pm}+\alpha =\widetilde H_{\pm}$ for every $\alpha\in \Gamma(p)$, that is, $\widetilde H_\pm$ 
% satisfy the strong Birkhoff property\footnote{Notice that, up to a sign change, it is not restrictive to assume that the elements of such basis satisfy 1.}. Since $x\in \overline{\widetilde H}_-\cap (\widetilde H_+)^c$  and thus $x\in \partial {\widetilde H_-}\cap \partial \widetilde H_+$, we get a contradiction.  

Assume for instance that $H_-\subset H_+$, and
let $x\in \partial H_-\cap \partial H_+\cap K$.
Let also $(\alpha_1,\dots,\alpha_k)\in \Gamma(p)$ be an integer basis
of vectors such that 1.~holds: $H_++\alpha_i\subset H_+$ and $H_-\subset H_-+\alpha_i$ for $i=1,\dots,k$. Observe that for any integer $n\ge 0$,
\[
n\alpha_i+H_-\ \subset\ (n+1)\alpha_i+H_-
\ \subset\ (n+1)\alpha_i+H_+\ \subset \ H_+\,.
\]
Hence, letting $\widetilde{H}=\bigcup_{n_i\ge 0}(H_-+\sum_{i=1}^k n_i\alpha_i)$,
we obtain a plane-like minimizer such that
$\widetilde{H}\subset H_+\subset E_+$,
$x\in \overline{\widetilde{H}}$, and $\widetilde{H}$
satisfies the strong Birkhoff property: hence $\widetilde{H}=E^+$. It follows
that $H_+=E^+$, in contradiction with the fact that $K\cap\partial H_+\neq
\emptyset$.\end{proof}
% Then 
% \[
% x\in \overline{H}_-\subset \alpha+\overline{H}_-\subset \alpha + \overline{H}_+ \qquad \textrm{and} \qquad  x\in H_+^c\subset \alpha+ H_+^c\subset \alpha + H_-^c
% \]
% so that, letting 
% \[
% H_-^\alpha:= \bigcup_{n\ge 0} (n\alpha + H_-) \qquad \textrm{and } \qquad H_+^\alpha:=\bigcap_{n\ge 0} (n\alpha + H_+),
% \]

\section{Differentiability and strict convexity of $\p$}\label{secpropphi}

We are now ready to prove our main result.

\begin{theorem}\label{main}
Let $F\in\mathcal{C}^{2,\alpha}(\R^d\times (\R^d\setminus\{0\}))$
be a convex, one-homogeneous and elliptic integrand.  
Then, the associated stable norm $\p$ has the following properties:
\begin{itemize}
\item $\phi^2$ is strictly convex;
\item if $p$ is totally irrational then $\nabla\p(p)$ exists;
\item the same occurs for any $p$ such that the plane-like minimizers satisfying the strong Birkhoff property  give rise to a foliation of $\R^d$;	
\item if, on the other hand, these minimizers form a lamination with a gap, then
$\partial \phi(p)$ is a convex set of dimension $\dim(V^r(p))$:
$\phi$ is differentiable in the directions of $(V^r(p))^\perp$ and is non-differentiable in the directions of $\R^d\setminus (V^r(p))^\perp$. 
%%In particular if $p$ is not totally irrational then $\p$ is not differentiable at $p$.
\end{itemize}
\end{theorem}

\subsection{Strict convexity}
\begin{theorem}\label{convex}  
The function $\phi^2$ is strictly convex.
%or, equivalently, the set $\{\p(p)\le 1\}$ is strictly convex.
\end{theorem}

\begin{proof}
Let $p_1,p_2$, with $p_1\neq p_2$, and  
let $p=p_1+p_2$. We want to show that, if $\p(p)=\p(p_1)+\p(p_2)$, then 
$p_1$ is proportional to $p_2$, which gives the thesis. 

Indeed, we have
\begin{align*}
\p(p) &= \int_\T [z_p, p+Dv_p]\\
&= \, \int_\T F(x,p+Dv_{p})\\
&\le\, \int_\T F(x,p+Dv_{p_1}+Dv_{p_2}) \\
&\le\, \int_\T F(x,p_1+Dv_{p_1})+F(x,p_2+Dv_{p_2})\\ 
&=\, \p(p_1)+\p(p_2)\,.
\end{align*}
Since $\p(p)=\p(p_1)+\p(p_2)$, it follows that $v_{p_1}+v_{p_2}$ is also a minimizer of \eqref{defphi} and,
%(hence we can take $v_p=v_{p_1}+v_{p_2}$ in the previous equation), 
in particular, $z_p$ satisfies
\[
[z_p,p_1+Dv_{p_1}]+[z_p,p_2+Dv_{p_2}]\ =\ 
F(x,p_1+Dv_{p_1})+F(x,p_2+Dv_{p_2})
\]
$(|p_1+Dv_{p_1}|+|p_2+Dv_{p_2}|)$-a.e., so that 
\[
[z_p, p_i+Dv_{p_i}]\ =\ F(x,p_i+Dv_{p_i}) \qquad i\in\{1,2\}\,.
\]
This means that $z_p$ is a calibration for the plane-like minimizers 
\[
\{v_p +p\cdot x\ge s\}\ ,\ \{v_{p_1}+p_1 \cdot x\ge s\}\ \textup{ and }\ \{v_{p_2} +p_2\cdot x\ge s\}
\]
for all $s\in\R$. By Proposition \ref{order}, it follows that they are included in one another
which is possible only if $p_1$ is proportional to $p_2$.
\end{proof}

\begin{remark}\rm
Observe that $\phi^2$ may fail to be uniformly convex.
The thesis of the theorem is
equivalent to the strict convexity of the level sets of $\phi$.
\end{remark}

%\begin{remark}\rm
% The difference between the simplicity of our proof and Senn's \cite{Sennstrict} proof in the non-paramteric case is quite striking. 
%\end{remark}

\subsection{Differentiability of $\p$}
We now turn to the study of the differentiability of $\p$. As already noticed, the differentiability of $\p$ at a point $p\in \R^d$ is equivalent to the fact that $\partial \p(p)$ is a singleton, that is, for every calibration $z\in X$ in the direction $p$ the integral $\int_Q z \ dx$ has the same value.

{ } Let us first show that $\phi$  must is differentiable in the totally irrational directions.
 
\begin{proposition}\label{totir}
Assume $p$ is totally irrational.
Then for any two calibrations $z,z'$ in the direction $p$,
$\int_Q z\,dx=\int_Q z'\,dx$. As a consequence, 
$\partial \p(p)$ is a singleton and $\p$ is differentiable at $p$.
\end{proposition}
\begin{proof}
Consider $z,z'$ two calibrations and a solution $v_p$ of~\eqref{defphi},
and let $\xi=\int_Q z(x)\, dx$, $\xi'=\int_Q z'(x)\,dx$.

{ } Let us show that, for any $s$,
\begin{equation}
\int_{\{v_p+p\cdot x=s\}} (z(x)-z'(x))\,dx\ =\ 0.
\end{equation}
Thanks to the density estimate, the level sets
$\{x\,:\, v_p(x)+p\cdot x=s\}$ are equivalent (up to a negligible set)
to an open set $C_s$ with 
$\partial C_s = \partial \{v_p+p\cdot x>s\}\cup \partial \{v_p+p\cdot x\ge s\}$.
Moreover, all the $C_s$ are empty except for a countable number of values.
Consider such a value $s$. Since $z$ and $z'$ calibrate $C_s$ which is a   plane-like minimizer we have $[z,\nu^{C_s}]=[z',\nu^{C_s}]$ on $\partial^* C_s$.

{ } Then, we observe that the sets $C_s^q=Q\cap (C_s-q)$, $q\in\Z^d$, are
all disjoint, so that their measures sum up to less than $1$.
Indeed, if there is a point $y\in (C_s-q_1)\cap (C_s-q_2)$
with $q_1,q_2\in\Z^d$,
then we have $v_p(y+q_1)+p\cdot y+p\cdot q_1=s=v_p(y+q_2)+p\cdot y
+p\cdot q_2=s$, but since $v_p(y+q_1)=v_p(y+q_2)$ it follows that
$p\cdot (q_2-q_1)=0$, hence $q_1=q_2$, since $p$ is totally irrational.

{ } For $R\in \N^*$,
let $\Psi_R$ be a Lipschitz cutoff function equal to $1$ on $[-R,R]^d$,
to $0$ out of $[-(R+1),R+1]^d$, and with $|\nabla \Psi_R|=1$
a.e.~in $K_R=[-(R+1),R+1]^d\setminus [-R,R]^d$. Recalling that
$\Div (z-z')=0$, we compute
\begin{equation}\label{psiRR}
\int_{C_s} \Psi_R(x)(z(x)-z'(x))\cdot e_i\,dx
\,=\, -\int_{C_s} x_i (z(x)-z'(x))\cdot\nabla \Psi_R(x)\,dx,
\end{equation}
which is bounded by
\begin{equation}
L_R \ =\ 2c_0^{-1} (R+1)|K_R\cap C_s|,
\end{equation}
where $c_0$ is the constant in~\eqref{boundF}.
Since $\sum_{R\ge 1} |K_R\cap C_s|\le 1$, we get
$\liminf_{R\to \infty} L_R=0$, otherwise we would
have $|K_R\cap C_s|\gtrapprox c/(R+1)$ for large $R$
and for some constant $c>0$, which would imply $\sum_R |K_R\cap C_s|=+\infty$.
Hence, there exists a subsequence $R_k\to +\infty$ with $L_{R_k}\to 0$,
but then, passing to the limit in~\eqref{psiRR} along this subsequence, we get
\begin{equation}
\int_{C_s} (z(x)-z'(x))\cdot e_i\,dx\,=\,0\,,
\end{equation}
which gives our claim.

We deduce that $\xi=\xi'$. Indeed, from Theorem \ref{zfix} it follows
that $$\int_{\R^d}(z-z')\,dx =\sum_{s} \int_{C_s} (z-z')\,dx=0\,,$$ 
where the sum is on all $s$ such that $C_s$ is an open, nonempty set.
In particular, we obtain that $\int_Q (z-z')\,dx=\xi-\xi'=0$.
\end{proof}

{ } If $p$ is not totally irrational, by taking the quotient of $\R^d$ with respect to all rational directions orthogonal to $p$, 
by the same argument we get the following result.

\begin{corollary}
For every $q \in (V^r(p))^\perp$, the function $\p$ 
is differentiable at $p$ in the direction $q$. 
This amounts to say that, for every $\xi_1, \xi_2 \in \partial \p(p)$ and every $q\in (V^r(p))^\perp$,
\[\xi_1 \cdot q= \xi_2 \cdot q.\]
\end{corollary}

As a direct consequence of Theorem \ref{zfix} we also have
\begin{proposition}\label{cinquepuntosei}
If the   plane-like minimizers of $\CA(p)$ fibrate $\R^d$ then $\phi$ is differentiable at $p$.
\end{proposition}
We finally investigate the non-differentiability of $\p$ at points
$p$ which are not totally irrational, and such that there is $x\in \R^d$
such that no minimizer passes through $x$. In that case let $G$ be a gap containing $x$,
bounded by two plane-like minimizers $E^\pm$.

{ } We start by investigating what happens for a rational vector $q\in \Gamma(p)=V^r(p)\cap \Z^d$. For such a $q$, let $H_q$ be a heteroclinic solution in the  direction $q$ with $z_1$ an associated calibration. Let also $z_2$ be a calibration associated to $H_{-q}$, the heteroclinic solution in the direction $-q$. We will prove that

 \[\int_Q z_1\cdot q \neq \int_Q z_2\cdot q\,.\]

Let us notice that by Theorem \ref{calibrallPL} $z_1$ and $z_2$ have the same normal component on the boundaries of every gap, moreover by Theorem \ref{zfixPL} they agree outside the gaps. 
Therefore, we can assume that $z_1$ and $z_2$ differ only inside the projection on the torus of the gap $G$, and thus we are reduced to prove that
\[\int_{Q\cap \Pi(G)} (z_1-z_2)\cdot q \neq 0,\]
where $\Pi: \R^d\to Q$ denotes the projection on $Q$. 
We consider a further decomposition of the space $V^r(p)$. 
By a Grahm-Schmidt procedure, we see that $V^r(p)$ is spanned by a family $(q,q_2, \dots, q_k)$ of orthogonal vectors in $\Z^d$.  
Let $\widetilde \Gamma:=\Span_{\Z}(q_2, \dots, q_k)$, 
%%and let $\sim$ be the equivalence relation defined as
%%\[ 
%%x \sim y\quad  \Longleftrightarrow \quad \exists\, \alpha \in \widetilde \Gamma \ \textrm{s.t.}\  y=x+\alpha
%%\]
%%for all $x,y \in  \R^d$.
and let $\widetilde \T_r:= \R^d / \widetilde\Gamma $ 
be the cylinder quotient of $\R^d$ and $\widetilde\Gamma$.
%%with respect to this equivalence relation.
Since the gap $G$, the sets in $\CA(p)$, and the heteroclinic plane-like minimizers in direction $q$ are periodic with respect to vectors in $\widetilde \Gamma$, 
we identify them with their quotient with respect to the group
action of $\widetilde\Gamma$.
%the equivalence relation $\sim$. 
Furthermore, since the $q_i$ are orthogonals, the cylinder $\widetilde \T_r$ can be identified with 
$\bigcap_{i=2}^k \left\{x\cdot \frac{q_i}{|q_i|^2} \in [0,1[\right\}$.  
Let 
\[\S^t_s:=\left\{ x \in \widetilde \T_r \,:\, s< x \cdot \frac{q}{|q|^2} < t\right\}
\]
and let $\S:=\S_0^1$ be the unit slab in the direction $q$. We will show that 
\begin{equation}\label{integzslab}
\int_{\S \cap G} (z_1-z_2)\cdot q \neq 0.
\end{equation}
Notice that there exist vectors $\alpha\in \widetilde \T^r$ such that $\alpha\cdot q=0$. 
However, there is only a finite number of vectors $\alpha\in\Gamma(p)$ such that $\alpha\in \bigcap_{i=2}^k \{x\cdot \frac{q_i}{|q_i|^2} \in [0,1[\} $ and $\alpha\cdot \frac{q}{|q|^2} \in [s,t]$  which   implies that $|\S^t_s\cap G| $ is finite for every $(s,t) \in \R$.
We also let 
\[S_t:= \left\{ x\in\widetilde \T^r  \,:\,  x \cdot \frac{q}{|q|^2} = t\right\}
\]
and $S:= S_0$. Since 
\[|\S^t_s\cap G| =|q|\int_s^t \HH(G\cap S_\tau) \ d\tau,\]
the measure $\HH(G\cap S_t) $ is finite for almost every $t\in \R$. 

{ } In particular, without loss of generality, we can assume that $\HH(G\cap S)$ is finite.

\begin{proposition}\label{egalmultipl}
\[
\int_{\S \cap G} (z_1-z_2)\cdot q=N\int_{Q\cap \Pi(G)} (z_1-z_2)\cdot q
\]
for some $N\in \N$, with $N\le C |q| \prod_{i=2}^k |q_i|$ for some $C>0$.
\end{proposition}
\begin{proof}
We will prove that $ \int_{\S \cap G} (z_1-z_2)\cdot q$ is an entire multiple of $ \int_{Q\cap \Pi(G)} (z_1-z_2)\cdot q $. Notice first that 
$$
\S=\left\{x\cdot \frac{q}{|q|^2} \in [0,1[\right\}\cap \bigcap_{i=2}^k \left\{x\cdot \frac{q_i}{|q_i|^2} \in [0,1[\right\}.
$$
Moreover, we have  
\[
\S\cap G=\bigcup_{\alpha\in \Z^d \,:\, (Q+\alpha)\cap \S \neq \emptyset} (\S\cap G)\cap (Q+\alpha).
\]
Let $\Gamma_\S:=\{\alpha\in \Gamma(p) \,:\, Q+\alpha \cap \S \neq \emptyset\}$. 
By the strong Birkhoff property, for every $\alpha_1,\alpha_2\in \Z^d$, if $\alpha_1\cdot p\neq \alpha_2\cdot p$ 
then $(G+\alpha_1)\cap (G+\alpha_2)=\emptyset$, and for every $\alpha_1,\alpha_2\in \Gamma_\S$,  $G+\alpha_1= G+\alpha_2$. 
It is therefore sufficient to prove that, for some $N\in \N$, we have
\[ \sum_{\alpha\in \Gamma_\S}\int_{ \S\cap G\cap (Q+\alpha)} (z_1-z_2)\cdot q\ =\ N\int_{Q\cap G} (z_1-z_2)\cdot q.  \]
Let  $\alpha\in \Gamma_\S$ and  $Q_\alpha:= Q+\alpha$. If $Q_\alpha\cap \partial \S= \emptyset$, since $G+\alpha=G$, we have $G\cap \S \cap Q_\alpha=(G\cap Q)+\alpha$ and thus 
\[
\int_{G\cap \S \cap Q_\alpha} (z_1-z_2)\cdot q= \int_{G\cap Q} (z_1-z_2)\cdot q. 
\]
If $Q_\alpha \cap \partial \S \neq \emptyset$, we assume first that $Q_\alpha$ intersects $\partial \S$ only on one of the facets of $\S$.  
Then, there exists $\tilde q \in (q,q_2, \dots, q_k)$ such that
\[
Q_\alpha \cap \S =Q_\alpha \cap \left\{ x\cdot \frac{\tilde q}{|\tilde q|^2} \ge 0\right\} 
\qquad \textrm{or} \qquad Q_\alpha \cap \S =Q_\alpha \cap \left\{ x\cdot \frac{\tilde q}{|\tilde q|^2} < 1\right\}.
\]
Assume that the first possibility holds (see Figure \ref{Rk}), then  $\tilde q$ is such that 
\[(Q_\alpha+\tilde q) \cap \S=(Q_\alpha+\tilde q) \cap \left\{ (x-\tilde q)\cdot \frac{\tilde q}{|\tilde q|^2} < 1\right\}
=(Q_\alpha+\tilde q)\cap\left(\left\{ x\cdot \frac{\tilde q}{|\tilde q|^2} < 1\right\}+\tilde q\right).\]

\begin{figure}[ht]
\centering{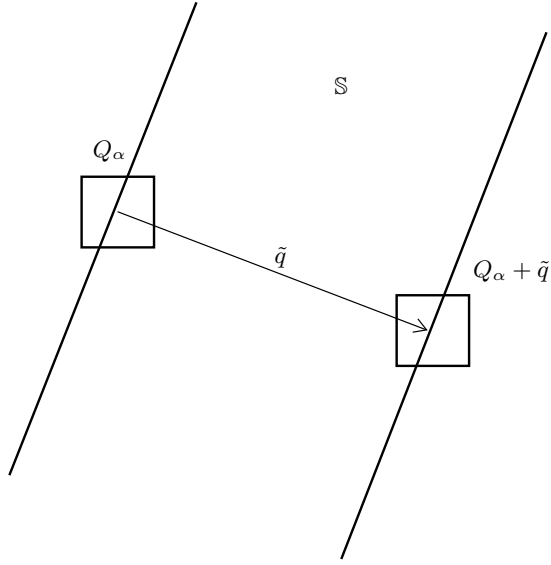}
\caption{The set $\S$ of Proposition \ref{egalmultipl}.}
\label{Rk}
\end{figure}

This shows that  
\[\left(Q_\alpha \cap \S\right)\cap \left([(Q_\alpha+\tilde q) \cap \S]-q\right)=\emptyset \qquad \textrm{and} \qquad 
\left(Q_\alpha \cap \S\right)\cup \left([(Q_\alpha+\tilde q) \cap \S]-q\right)=Q_\alpha\]
whence 
\[
\int_{Q_\alpha \cap \S\cap G} (z_1-z_2)\cdot q+\int_{(Q_\alpha+\tilde q)\cap \S\cap G} (z_1-z_2)\cdot q=\int_{Q\cap G} (z_1-z_2)\cdot q.
\]
If $Q_\alpha$ intersects $m$ facets of $\S$, then we can assume that, for some $\tilde q_1, \dots, \tilde q_m \in (q,q_2, \dots, q_k)$,
\[
Q_\alpha\cap \S =\bigcap_{j=1}^m \left\{ x\cdot \frac{\tilde q_j}{|\tilde q_j|^2} \ge 0	\right\}.
\]
We can then repeat the above argument by pairing the cube $Q_\alpha$ with the $2^m$ cubes of the form $Q_\alpha+\sum_{j=1^m} \delta_{j} \tilde q_j$, 
where $\delta_j$ takes only values $0$ or $1$. This proves that,  for some $N\in \N$,
\[ \sum_{\alpha\in \Gamma_\S}\int_{ \S\cap G\cap (Q+\alpha)} (z_1-z_2)\cdot q\ =\ N\int_{Q\cap G} (z_1-z_2)\cdot q.  \]
Moreover $N\le \# \Gamma_\S\le C|q|  \prod_{i=2}^k |q_i|$.
\end{proof}

\begin{proposition}
For almost every $s, t \in \R$, we have
\begin{equation}\label{segt}
\int_{S_t\cap G} \left[z_1-z_2, \frac{q}{|q|}\right]=\int_{S_s\cap G} \left[z_1-z_2, \frac{q}{|q|}\right].
\end{equation}
In particular,
\begin{equation}\label{egsurf}\int_{\S \cap G} (z_1-z_2)\cdot q =|q|^2 \int_{S\cap G} \left[z_1-z_2, \frac{q}{|q|}\right].
\end{equation}
\end{proposition}

\begin{proof}
Fix $s<t \in \R$ such that  $\HH(G\cap S_t) $  and $\HH(G\cap S_s) $ are finite. Let $\psi_R$ be the same cutoff function as in Proposition \ref{totir}, then 
\begin{align*}
 \int_{\S_s^t\cap G} \Div (\psi_R (z_1-z_2)) &= \int_{\S_s^t\cap G} \nabla \psi_R \cdot (z_1-z_2)\\
					&=\int_{S_t\cap G} \psi_R [z_1-z_2, q]-\int_{S_s\cap G} \psi_R [z_1-z_2, q].
\end{align*}
As in Proposition \ref{totir}, $ \int_{\S_s^t\cap G} \nabla \psi_R \cdot (z_1-z_2) \rightarrow 0$ and thus letting $R\to \infty$ we find \eqref{segt}. By \cite[Prop. 2.7]{Anzelotti} applied to $u(x)=q\cdot x$, we have 
\[\int_{\S \cap G} (z_1-z_2)\cdot q=\int_0^{|q|^2} \left(\int_{S_t\cap G} \left[z_1-z_2, \frac{q}{|q|}\right]\right) \ dt\]
which gives \eqref{egsurf}.
\end{proof}

\begin{proposition}
Let $\nu$ be the inward normal to $H_q$, then
 \begin{equation}\label{egheter}\int_{S\cap G} \left[z_1-z_2, \frac{q}{|q|}\right]= 
 \int_{\partial^* H_q\cap \widetilde \T_r} [z_1-z_2, \nu].\end{equation}

\end{proposition}
\begin{proof}
We first introduce some additional notation (see Figure \ref{sigmaplus}): let 
\[\Sigma^+:= \partial^* H_q \cap \{ x \cdot q >0\}\cap \widetilde \T_r \qquad G^+:= G\cap \{ x \cdot q >0\} \cap H_q^c\cap \widetilde \T_r\]
and 
\[S^+:= S\cap G \cap H_q^c.\]

\begin{figure}[ht]
\centering{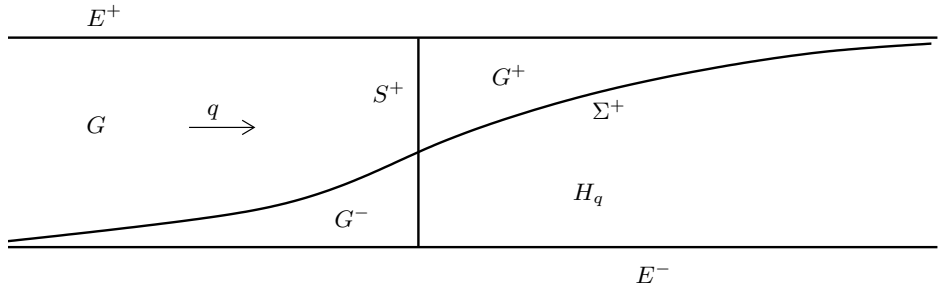}
\caption{Heteroclinic solution in the direction $q$.}
\label{sigmaplus}
\end{figure}

Then, letting $\psi_R$ be the usual cutoff function, we have
\[
 \int_{G^+} \Div (\psi_R (z_1-z_2))=\int_{S^+} \psi_R \left[z_1-z_2, \frac{q}{|q|}\right] -\int_{\Sigma^+} \psi_R [z_1-z_2, \nu]
\]
As usual 
\[ \int_{G^+} \Div (\psi_R (z_1-z_2)) \to 0 \qquad \textrm{and} \qquad 
\int_{S^+} \psi_R \left[z_1-z_2, \frac{q}{|q|}\right]\to \int_{S^+} (z_1-z_2)\cdot \frac{ q}{|q|}\]
On $\Sigma^+$, $[z_1,\nu]=F(x,\nu)\ge [z_2,\nu]$ hence $\psi_R [z_1-z_2,\nu]$ converges monotonically to $[z_1-z_2,\nu]$  and thus passing to the limit when $R\to +\infty$ we get by the Monotone Convergence Theorem, 
\[\int_{\Sigma^+} [z_1-z_2, \nu]= \int_{S^+} \left[z_1-z_2, \frac{q}{|q|}\right].\]
Similarly we define $\Sigma^-$ and $S^-$ and get 
\[\int_{\Sigma^-} [z_1-z_2, \nu]= \int_{S^-} \left[z_1-z_2, \frac{q}{|q|}\right].\]
Summing these two equalities we find \eqref{egheter}.
\end{proof}

% \begin{proposition}
%  \[\int_{\partial H_q} [z_1-z_2, \nu] >0\]
% \end{proposition}
% 
% \begin{proof}
%  Since $z_1$ calibrates $\partial H_q$, $[z_1, \nu] = F(x,\nu)$ thus 
%  \[\int_{\partial H_q} [z_1-z_2,\nu]=\int_{\partial H_q} F(x,\nu)-[z_2, \nu] \ge 0.\]
% If $\int_{\partial H_q} [z_1-z_2, \nu] =0$ then $[z_2, \nu] =F(x,\nu)$ and thus $z_2$ calibrates also $\partial H_q$ which is impossible since it calibrates the heteroclinic solution in the direction $-q$ which intersects $\partial H_q$.
% 
% \end{proof}

 We then have : %(where as usual we have taken the quotient with respect to the vectors in $V^r(p)\cap \Z^d$ which are orthogonal to $q$)

\begin{proposition}
There exists a constant $C$ such that, if we take $q\in \Gamma(p)$ and we let $(q, q_2, \dots, q_k)\in \Z^q$ be an orthogonal basis of $V^r(p)$ constructed as above, then for every heteroclinic surface $H_q$ in the direction $q$ and every $z\in X$ which calibrates a heteroclinic surface $H_{-q}$ in the direction $-q$ there holds
\begin{equation}\label{egalmultiplheter}\int_{\partial^* H_q\cap \widetilde \T_r} F(x, \nu^{H_q}) -[z,\nu^{H_q}]\ge C\delta \prod_{i=2}^k |q_i|.\end{equation}
\end{proposition}

\begin{proof}
  Let $K$, $R$ and $\delta$ be as in Proposition \ref{estimdelta}, so that whenever $\partial^* H_{q}\cap \partial^* H_{-q}\cap K\neq \emptyset$ we have
\[\int_{\partial^* H_q\cap B_R} F(x,\nu^{H_q})-[z,\nu^{H_q}]\ge \delta. \]
For every $\alpha\in \Gamma(p)$, the same estimate holds for the set $B_R+\alpha$. If $(\bar q_1, \dots, \bar q_k)$ is the basis of $V^r(p)$ that we have fixed from the beginning, letting $\Gamma_R:=\Span_\Z(2\lceil R\rceil \bar q_1, \dots, 2\lceil R\rceil \bar q_k)$ we have  $(B_R+\alpha) \cap (B_R+\beta)=\emptyset$ for $\alpha, \beta \in \Gamma_R$ and thus
\[\int_{\partial^* H_q\cap \widetilde \T_r} F(x, \nu^{H_q}) -[z,\nu^{H_q}]\ge \sum_{\stackrel{\alpha \in \Gamma_R}{B_R+\alpha\subset \widetilde \T_r}} \int_{\partial^* H_q\cap (B_R+\alpha)} F(x, \nu^{H_q}) -[z,\nu^{H_q}].\]
The number of vectors  $\alpha\in \Gamma_R$ such that $\partial^* H_q\cap \partial^* H_{-q}\cap (B_R+\alpha)\neq \emptyset$ and $B_R+\alpha \subset \widetilde \T_r$, is proportional to $\prod_{i=2}^k |q_i|$. Since for every such $\alpha$, there holds 
\[ 
\int_{\partial^* H_q\cap (B_R+\alpha)} F(x, \nu^{H_q}) -[z,\nu^{H_q}]\ge\delta\,,
\]
we get \eqref{egalmultiplheter}.

%Since we can always assume that $\partial^* H_q\cap \partial^* H_{-q} \cap K\neq \emptyset$ where $K$ is the set defined in Section \ref{sechetero} as $K=\bigcap_{i=1}^d \{0\le x\cdot \frac{\bar q_i}{|\bar q_i|^2}<2\} $. Proposition \ref{estimdelta}, then implies 
% \[\int_{K\cap \partial^* H_q} F(x, \nu^{H_q})-[z,\nu^{H_q}] \ge \delta.
% \]
% If we denote by $\lambda:=\#\{k\in \Z^d\ / \ Q+k\cap K \neq \emptyset\}$ then there is at least one entire translate $\tilde Q$ of the unit cube such that $\tilde Q \cap K \neq \emptyset$ and 
% 
% \[\int_{\tilde Q\cap \partial^* H_q} F(x, \nu^{H_q})-[z,\nu^{H_q}] \ge \frac{ \delta}{\lambda}.
% \]
% 
%  Now if $k\in \Gamma(p)$ with $k\cdot q=0$, since $H_q$ and is invariant by the translation of vector $k$ we have
% \[\int_{(\tilde Q+k)\cap \partial^* H_q} F(x, \nu^{H_q})-[z,\nu^{H_q}]= \int_{\tilde Q\cap \partial^* H_q} F(x, \nu^{H_q})-[z,\nu^{H_q}]\ge  \frac{ \delta}{\lambda}.
% \]
% Arguing as in Proposition \ref{egalmultipl}, we see that 
% \[\int_{\partial^* H_q\cap \tilde \T_r}  F(x, \nu^{H_q})-[z,\nu^{H_q}]\ge C N_2\int_{Q\cap \partial^* H_q} F(x, \nu^{H_q})-[z,\nu^{H_q}]\ge C N_2 \delta \]
% where $N_2$ is the number of vectors in $\Gamma(p)\cap \bigcap_{i=2}^k \{x\cdot \frac{q_i}{|q_i|^2} \in [0,1[\} $ with $q\cdot k=0$. The number $N_2$ is thus proportional to $\prod_{i=2}^k |q_i|$ which ends the proof.
\end{proof}

\begin{proposition}\label{finalnondiff}
 Let $q \in V^r(p)\setminus\{0\}$ then there exists $\xi_1$ and $\xi_2$ in $\partial \p(p)$ such that $\xi_1\cdot q \neq \xi_2\cdot q$. As a consequence, $\partial \p(p)$ is a convex set of dimension $\dim(V^r(p))$.  
\end{proposition}

\begin{proof}
 Let us start with $q\in \Gamma(p) \setminus\{0\}$. Take $\xi_1:=\int_Q z_1$ and $\xi_2:=\int_Q z_2$. If $(q_2, \dots,q_k)\in \Z^d$ is such that $(q,q_2, \dots,q_k)$ is an orthogonal basis of $V^r(p)$ then using the same notation as before, by  Proposition \ref{egalmultipl}, 
\[\int_Q (z_1-z_2)\cdot q=\int_{Q\cap \Pi(G)} (z_1-z_2)\cdot q = \frac{1}{N} \int_{\S\cap G}(z_1-z_2)\cdot q\]
with $N\le C |q|\prod_{i=2}^{k} |q_i|$.
By \eqref{egsurf} we get
\[\int_Q (z_1-z_2)\cdot q= \frac{|q|^2}{N} \int_{S\cap G} \left[z_1-z_2, \frac{q}{|q|}\right].\]
Equation \eqref{egheter} then yields
\[\int_Q (z_1-z_2)\cdot q= \frac{|q|^2}{ N} \int_{\partial^* H_q\cap \widetilde \T_r} \left[z_1-z_2, \nu^{H_q}\right]\ge \frac{1}{C \prod_{i=2}^{k} |q_i|} |q|\int_{\partial^* H_q\cap \widetilde \T_r} [z_1-z_2, \nu^{H_q}] .\]
Using finally \eqref{egalmultiplheter} we find
\begin{equation}\label{estimsub}\int_Q (z_1-z_2)\cdot q\ge C \delta |q|,\end{equation}
for some $C>0$, which gives the desired result for $q\in \Z^d$. If now $q=\alpha/\beta \in \mathbb{Q}^d$ with $\alpha \in \Z^d$ and $\beta\in \N^*$ then applying \eqref{estimsub} to $\beta q$, by homogeneity  \eqref{estimsub} holds also for $q$. Finally, every vector of $V^r(p)\setminus\{0\}$ can be approximated by rational vectors and since the lower bound of \eqref{estimsub} is uniform, it passes to the limit. 
\end{proof}

Theorem \ref{main} now follows from Theorem \ref{convex}, Proposition \ref{totir}, Proposition \ref{cinquepuntosei}, Proposition \ref{finalnondiff}, recalling the fact that the directions of differentiability of $\p$ at a point $p$ form a linear subspace of $\R^d$.
%\subsection{Some trash which remains}
%For any $s\in\R$ we can define 
%\[
%\Gamma_s=\bigcup_{k\in\Z^d} \partial \left\{ x\in Q\,:\,
%p\cdot x+v_p(x)> s-k\cdot p \right\}
%\]
%as the fibration on the torus of the   plane-like minimizer $\{v+p\cdot x>s\}$.
%First we observe that through any $x\in \overline{\Gamma}_s$, there
%is a   plane-like minimizer $E$ with $x\in \partial E$. Indeed,
%there exists $x_n\to x$ and $E_n=\{ p\cdot y+v_p(x)>s-k_n\cdot p\}$
%such that $x_n\in\partial E_n$. The sets $E_n$ (which satisfy
% uniform density estimates) are compact both in $L^1_\loc$ and
%in the Kuratowski sense (as well as their complements),
%hence, up to a subsequence they converge to
%a limit $E$ with $x\in \partial E$. It is then not difficult
%to show that $E$ is a   plane-like minimizer (by stability, see~\cite{CDLL,CT}).
%
%The same is obviously true for $\overline {\bigcup_{s} \Gamma_s}$.
%Now, both $z$ and $z'$ ``calibrate'' the   plane-like minimizers.
%%%%%%
%\section{Strict convexity}

\section{Some examples}\label{examples}

Whether or not gaps do occur in laminations by plane-like minimizers is a delicate question. In \cite{Bangexist}, Bangert proved that for every  Riemannian metrics (i.e. if $F(x,p)=a(x)|p|$ and $g=0$), for every periodic open set $V\subset \R^d$, it is possible to modify the function $a(x)$ inside $V$ in such a way that for every direction there always exists a gap. A simple example of a functional for which there are gaps in every direction can be constructed as follows. In dimension two, let $a(x):=c_1$ in a square $D$ strictly contained in $Q$, and $a(x):= c_2$ outside $D$, with $c_1>\sqrt{2}c_2$. Then, any plane-like minimizer must contain or avoid $D$, so that there are gaps in every directions. Notice that by making a little regularization, it is possible to have the same behavior for a functional satisfying the hypotheses of Theorem \ref{main}. In \cite{gilles} the author constructs an example of a functional 
with gaps in every lamination by plane-like minimizers, by considering  the prescribed mean curvature functional (i.e. $F(x,p)=|p|$ and $g\neq 0$)  with $g$ equal to some $\lambda$ in a small ball and $-\lambda$ in another ball. In the next example, the function $\p$ is differentiable in every direction except from one. 

{ } Let $d=2$ and $F(x,y,\nu)=a(x)|\nu|$, with $a(x)$ a periodic function of the first variable (for example $a(x)= \sin(\frac{x}{2\pi})+2$). Then, for $p=\pm(1,0)$, the only plane-like minimizers in the direction $p$ are the planes orthogonal to $p$ passing through the minima of $a(x)$. 
Thus $\p$ is not differentiable at $\pm(1,0)$ whereas for every other $p\in \S^1$, if $E$ is a   plane-like minimizer in the direction $p$, 
by invariance of the functional by translation along the $y$ component, the set $E+t (0,1)$ is also a   plane-like minimizer 
for every $t\in \R$, therefore there is no gap in the direction $p$, and $\p$ is differentiable at $p$ (see Figure \ref{phisin}). 
Clearly, this construction can be done in any dimension.

\begin{figure}[ht]
\centering{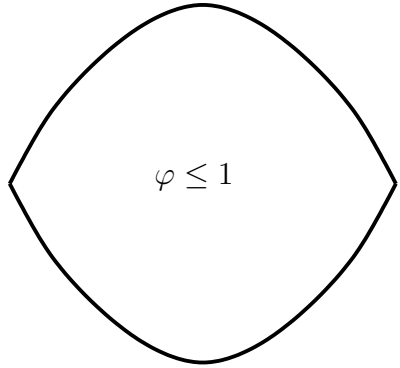}
\caption{Example with a gap for a translation invariant functional.}
\label{phisin}
\end{figure}
{ } One can wonder how non differentiability of $\p$ is related to the invariance by translation of the functional. The relation is not clear at all as shown by our last example where the energy  is not invariant in any direction but for which the associated $\p$ is differentiable in some rational direction.

{ } Indeed, let $\psi$ be a periodic smooth function on $\R^d$ with $\|\nabla \psi\|_{\infty}\le \frac{1}{2}$, let $u(x):=x_1+\psi(x)$, where $x=(x_1,x')$, 
and let $z:=\frac{e_1+\nabla \psi(x)}{|e_1+\nabla \psi(x)|}$ (the condition $\|\nabla \psi\|_{\infty}\le \frac{1}{2}$ is ensures that $z$ is well defined). The vector field $z$ is then normal to all level-sets of $u$, which fibrate $\R^d$.
Letting $g=\Div z$, we see that $z$ is a calibration of the level-sets of $u$ for the energy
\[P(E)+\int_E g\,dx\,.\]
We thus have found
a fibration of the space by plane-like minimizers in the direction $e_1$:
so that the corresponding $\p$ is differentiable along $e_1$. However,
in general $g$ will not be invariant in any direction.

\section{G-closure of isotropic perimeter functionals}\label{Gclosure}
A natural question is to determine the set of anisotropies that one can obtain by homogenization of the isotropic interfacial energies
\[\E(E,A)=\int_{\partial^*E\cap A} a(x) d \HH(x)\]
through the formula \eqref{defphi}. If no bound is imposed on $a(x)$ (that is no restriction is made on the $c_0$ of \eqref{boundF}), we can easily see that the set of such interfacial energies is dense in the set of all symmetric anisotropies $\p$, that is, the convex one-homogenous functions $\p$ with $\p(-p)=\p(p)$.
Indeed, it is sufficient to prove that we can obtain any crystalline energy with rational vertices (these are functions $\p$ for which the unit ball $\{\p \le1\}$ is a polytope whose vertices are rational points). Let $p_1, \dots, p_n$ be the rational  vertices of a given convex symmetric polytope $K$. Let $\Pi$ be the projection form $\R^d$ on the torus. By approximation we can consider a function $a(x)$ defined by
\[a(x)=\begin{cases}
        \lambda_i \qquad \textrm{if } x \in \Pi(\{p_i \cdot y=0\}) \textrm{ for some } i,\\
+\infty \qquad \textrm{otherwise}
       \end{cases}
\]
where $\lambda_i=|p_i|^{\frac{1}{d-1}}$, so that for $p=p_i$ the plane-like minimizers are given by the half-spaces $\{p_i \cdot x>0\}+q$ with $q \in \Z^d$, hence $\p(p_i)=1$ and it follows that $\{\p \le 1\}=K$. Notice that, when the hyperplanes $\{p_i\cdot x=0\}$ cross, $a(x)$ is not well defined but, since the intersection is of dimension $d-2$, it does not contribute to the energy. In Figure \ref{ax}, we show the simplest example where $\p$ is the $L^1$ norm, so that $K$ is the cube with vertices $\pm e_i$ for $e_i$ the canonical basis of $\R^d$. 

\begin{figure}[ht]
\centering{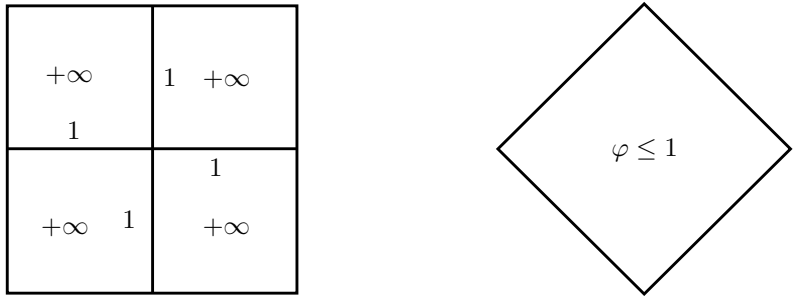}
\caption{Example of the construction of $\p=|\cdot|_1$. Left: the metric $a(x)$. Right: the unit ball $K$}
\label{ax}
\end{figure}
{ } By slightly modifying this construction, it is possible to approximate any even anisotropy $\p$ with $c_0 |p|\le\p(p)\le \frac{1}{c_0}|p|$ by isotropic interfacial energies $\E$ with Lipschitz functions $a(x)$ satisfying the same bounds, $c_0\le a(x)\le \frac{1}{c_0}$ (see \cite{AI}).

{ } This kind of  issues is very related to the famous G-closure problem for composite materials (see the recent paper \cite{BaBar}  and the references therein) or, in a setting closer to ours, 
to the density with respect to $\Gamma$-convergence of Riemannian metrics in the set of all Finsler metrics \cite{BrButFra}. 

\appendix
\section{A discrete ``separation'' result}\label{AppBir}

We show here the following result (see~\cite{UNIBan} for a slightly
more complex proof):
\begin{lemma}\label{lembirkhoffdir}
Assume $E$ is a set which satisfies the Birkhoff property, that is,
for any $q\in\Z^d$, either $q+E\subseteq E$ or $q+E\supseteq E$.
Then there exists $p\in \R^d$, $|p|=1$, such that for any $q\in \Z^d$,
$q\cdot p>0\Rightarrow q+E\subseteq E$ and (obviously)
$q\cdot p<0\Rightarrow q+E\supseteq E$. Moreover, $p$ is unique,
unless $E+q=E$ for all $q\in\Z^d$.
\end{lemma}
\begin{proof}
Let $Z=\{ q\in\Z^d\,:\, q+E\subset E\}$, clearly, $0\in Z$, $Z+Z=Z$
(so that $nZ\subset Z$ for any $n\ge 0$),
and the Birkhoff property states that for any $z\in\Z^d$,
either $z\in Z$ or $-z\in Z$. Without loss of generality we
may therefore assume that $e_i\in Z$ for $i=1,\dots,d$, where
$(e_i)$ is the canonical basis of $\R^d$.

We claim that either $Z=\Z^d$, or the closed convex envelope of $Z$
is not $\R^d$. In the latter case, since this envelope is a
convex cone, it must be contained in a semispace, hence the result.

Assume, thus, that $Z\neq \Z^d$ but any point in $\Z^d$ is
in the convexification of $Z$. In particular, it must be
that $-\sum_i e_i\not\in Z$, otherwise, for any $p=(p_1,\dots,p_d)\in \Z^d$
we would obtain that given a nonnegative integer $n\ge -\min_i p_i$,
$p=-n\sum_i e_i + \sum_i (p_i-n)e_i$ is also in $Z$.

A consequence is that as soon as $p_i<0$ for all $i$, then $p\not\in Z$,
otherwise $p+ \sum_i (-p_i-1)e_i = -\sum_i e_i \in Z$, which gives
another contradiction.

Now, by assumption, for any $\eps>0$ there exist $(p^k)_{k=1}^K$,
$(\theta^k)_{k=1}^K$ with $p^k\in Z$, $\theta^k\ge 0$,
$\sum_k \theta^k=1$ such that 
\[
\left| \sum_{k=1}^K \theta^k p^k + \sum_{i=1}^d e_i \right|\ <\ \eps\,.
\]
Possibly changing infinitesimally the $\theta^k$'s we can
assume that they are rational, hence, $\theta^k=n^k/m$ for
some integers $n^k\ge 0$, $m>0$, $\sum_k n^k=m$.

It follows, letting $p=\sum_k n^k p^k \in Z$,
\[
\left| p + m\sum_{i=1}^d e_i \right|\ <\ m\eps\,.
\]
As a consequence, for each $i$, $p_i< -m(1-\eps)\le 0$  as soon as $\eps<1$,
which implies that $p\not\in Z$, a contradiction. Hence, the closed
convex envelope of $Z$ is strictly contained in $\R^d$.
\end{proof}

\section{A generic uniqueness result}
\label{SecGeneric}

In Theorem~\ref{thmane} we have shown that the minimizer
of \eqref{defphi} is unique up to an additive constant, if the direction
$p$ is irrational. In addition, we show here that it is generically unique when $p$ is rational, 
that is, we prove a geometric counterpart of Ma\~n\'e's conjecture \cite{mane}.

For this we follow and adapt the proof of \cite{BC}. Contrary to the non-parametric case, it is no longer true that, if $F$ is an admissible anisotropy and $f\in C^{\infty}(\T)$, then $F(x,p)+f(x)|p|$ is also an admissible anisotropy. Indeed, if $\inf f \le -c_0$, the function $F(x,p)+f(x)|p|$ is not coercive anymore. We will thus restrict ourselves to positive perturbations. For this reason we cannot directly use \cite[Thm. 5]{BC} as in \cite{bessiMassart}. 

We will try to stay as close as possible to the notation of  \cite{BC}. 
A set $\O$ is called a residual set if it is a countable intersection of dense open sets. 
In a complete metric space, by Baire's Theorem, this implies that $\O$ is itself dense.

\begin{theorem}\label{maneDu}
For every rational vector $p$, there exists a residual set $\O(p)$ of $E:=C^{\infty}(\T)\cap\{f\ge 0\}$ such that for every $f\in \O(p)$, the minimizer of \eqref{defphi} for $F+f(x)|p|$ is unique up to an additive constant.
\end{theorem}

Following an idea of Mather, we first rewrite \eqref{defphi} as a linear problem. Notice that, for $u\in BV(\T)$, if we set $\mu_u:=  |Du+p|\otimes\delta_{\frac{Du+p}{|Du+p|}}$ (which is a Radon measure on $\T\times \S^{d-1}$) then
\[
\int_{\T} F(x,Du+p)\ =\ \int_{\T\times \S^{d-1}} F(x,\nu) d\mu_u
\]
is linear in $\mu_u$. Let $\widetilde H_r$ be the set of measures $\mu_u$ for $u\in BV(\T)$ with total variation less than $r$. 
Let $H_r$ be the weak-$*$ closure of the convex hull of $\widetilde H_r$. 
By Banach-Alaoglu's Theorem, $H_r$ is compact, convex and metrizable. 
Let $F$ be the space of Borel measures on $\T\times \S^{d-1}$, $G$ be the space of Borel measures on $\T$ and $K_r\subset G$ 
be the metrizable compact convex set of Radon measures on $\T$ with total variation less than $r$. 
Then, if $\pi: F\to G$ is the projection on the first marginal, for every $\mu_u\in \widetilde H_r$, 
$\pi(\mu_u)=|Du+p|\in K_r$ and thus $\pi(H_r)\subset K_r$. Letting 
\[
MA(F,\mu)\ :=\ \int_{\T\times \S^{d-1}} F(x,\nu) d\mu \qquad \mu\in H_r\,,
\]
we have that $MA(L,\cdot)$ is linear and continuous for the weak-$*$ topology on $H_r$. %% By the definition of $F$,  (cf \cite{CGNPL}),
If $u$ is a minimizer of \eqref{defphi} for $F+f$ with $f\in E$, we have 
\[
\left(\frac{1}{c_0}+|f|_\infty\right)|p| \ \ge\ \int_{\T} F(x,Du+p)+f(x)|Du+p|
\ \ge\ c_0 |Du+p|(\T)
\]
and thus for $r\ge (c_0^{-1}+|f|_\infty)|p|/c_0$ the measure $\mu_u$ is the minimizer of $MA(F+f,\cdot)$ in $\widetilde H_r$. Since $MA(F+f,\cdot)$ is linear, the minimum over $H_r$ and over $\widetilde H_r$ coincide. Hence, for every minimizing $u$, the measure $\mu_u$ is also a minimizer of $MA(F+f,\cdot)$ in $H_r$. We are thus going to prove that such minimizers in $H_r$ are generically unique.
Following \cite{BC}, let $\MHrF:=\argmin_{\mu\in H_r} MA(F,\mu)$ and $\MKF:=\pi(\MHrF)$.
\begin{proposition}
 For every $r>0$ there exists a residual set $\O_r\subset E$ such that for every $f\in \O_r$, the set $\MKFf$ is reduced to a single element. 
\end{proposition}

\begin{proof}
 Let now $\O(\eps)$ be the set of points $f\in E$ such that $\MKFf$ is contained in a ball of radius $\eps$ in $K_r$. We will prove that the Proposition holds for 
\[\O_r:=\bigcap_{\eps>0} \O(\eps).\]
Indeed, if $f\in \O_r$, and if $\MKFf$ is not a singleton then it is a convex set of positive dimension which would not be included in a ball of radius $\eps$ for $\eps $ small enough, contradicting the hypothesis $f\in \O_r$. It is thus enough to prove that for every $\eps>0$, the sets $\O(\eps)$ are open and dense.  
\end{proof}

The fact that $\O(\eps)$ is open is a direct consequence of the continuity of the map 
\[
(f,\mu)\ \to\ \int_{\T\times \S^{d-1}} F(x,\nu)+f(x)|\nu| d\mu
\]
which implies that for every open subset $U$ of $H_r$, the set $\{f\in E \, :\, \MHrFf\subset U\}$ is an open set of $E$ and similarly for $\MKFf$.
The density argument is more involved. Let $w\in E$, we want to prove that $w$ is in the closure of $\O(\eps)$. Repeating verbatim the proof of \cite[Lemma 7]{BC} there holds:
\begin{lemma}\label{Tm}
There exists an integer $m$ and a continuous map $T_m : \, K_r\to \R^m$ 
\[
T_m(\eta)\ := \ \left( \int_{\T} w_1 d\eta,\cdots, \int_{\T} w_m d\eta\right)
\]
with $w_i\in E$ and such that 
\[
\forall\, x\in \R^m \qquad \textup{diam} \, T_m^{-1}(x)<\eps
\]
where the diameter is taken for the distance on $K_r$.
\end{lemma}

Define the function $\Lambda_m: \R^m \to \R\cup\{+\infty\}$ as 
\[
\Lambda_m(x)\ :=\ \min_{\stackrel{\mu\in H_r}{T_m\circ \pi(\mu)=x}} MA(F+w,\mu)
\]
if $x\in T_m(\pi(H_r))$ and $+\infty$ otherwise. For $y=(y_1, \cdots , y_m)\in \R^m$, let 
\[
M_m(y)\ :=\ \argmin_{x\in \R^m} \, \Lambda_m(x)+y\cdot x
\]
then for $y\in \R^m_+$
\[
M_{K_r}(F+w+\sum_{i=1}^m y_i w_i)\ \subset\ T_m^{-1}(M_m(y)).
\]
Letting $\O_m:=\{y\in \R^m \, :\, M_m(y) \textrm{ is reduced to a point}\}$, we have from Lemma \ref{Tm}
\[
y\in \O_m \textrm{ and } y\in\R^m_+ \, \Longrightarrow \, w+\sum_{i=1}^m y_i w_i \in \O(\eps)
\]
it is thus enough to prove that $0$ can be approximated by positive vectors of  $\O_m$.

For this consider the convex conjugate of $\Lambda_m$,
\begin{align*}
 G_m(y)\ :=\ &\sup_{x\in \R^m} y\cdot x -\Lambda_m(x)\\
	=\ &\sup_{\mu\in H_r} MA\left(\sum_{i=1}^m y_i w_i - w-F,\mu\right).
\end{align*}
Since $H_r$ is compact, it is a convex and finite valued function which is then continuous on $\R^m$. Letting $\Sigma:= \{y\in \R^m \, : \, 
\textup{dim}\partial G_m(y)\ge 1\}$ we have that
$\textup{dim}\Sigma \le m-1$  (see \cite[App. A]{BC} or \cite{AA}) and thus the complement of $\Sigma$ is dense in $\R^m$. Since for every $y\in \R^m_+$ we have  $M_m(y)=\partial G_m(-y)$, it follows that for every $y\in \R_m^+\cap \Sigma^c$ the set $M_m(y)$ is reduced to a point, which proves the claim.

\medskip

We can finally end the proof of Theorem \ref{maneDu}. Let $\O(p):=\bigcap_{r>0} \O_r$ then by Baire's Theorem, $\O(p)$ is still a residual set of $E$. If $f\in \O(p)$ then for  $r\ge (c_0^{-1}+|f|_\infty)|p|/c_0$, the set  $\MKFf$ is reduced to a single element and if $u$ and $v$ are two different minimizers of \eqref{defphi} for $F+f$, it follows that $|Du+p|=|Dv+p|$. For $s\in \R$ and $E_s:=\{u+p\cdot x>s\}$, we can construct as in Proposition~\ref{Proprecurrent}, a minimizer $\widetilde u$ such that the levelsets of $\widetilde u +p\cdot x$ corresponds exactly to the projection $\Pi(\partial E_s)$ of $\partial E_s$ in the torus $\T$. Therefore, the measure $|Du+p|$ is concentrated on $\Pi(\partial E_s)$ and since on $\partial E_s$ there holds $\frac{Du+p}{|Du+p|}=\nu^{E_s}$, we find that $Du+p=D\widetilde u +p$ and hence $Du$ is unique. \qed

\begin{remark}\rm
From this uniqueness result, it can be easily proved that every plane-like minimizer with the Birkhoff property  generically induces the same current. 
Since it is not the main focus of our work and it would consist in repeating the arguments in \cite{bessiMassart}, we just sketch the proof. For $p\in \S^{d-1}$, $u$ a minimizer of the cell problem, $E$ a plane-like minimizer with the Birkhoff property and $v\in C^{\infty}(\T)$ a periodic vector field, define the currents $T_u$ and $T_E$ by
\[T_u(v):= \int_{\T} v\cdot (Du+p) \qquad \textrm{and} \qquad T_E(v):=\lim_{R\to +\infty} \frac{1}{\HH(\partial B_R)} \int_{\partial^*E \cap B_R} v\cdot \nu^E  \, d\HH,\] 
where the limit defining $T_E$ exists arguing as in \cite{bessi}. If $Du$ is unique and $E$ is recurrent then $T_u=T_E$. Now, for every plane-like minimizer $E$ with the Birkhoff property and every $q\in \Z^d$, we have $T_{E+k}=T_E$ (see \cite[Lem. 3.1]{bessiMassart}). Thus, letting $\widetilde E:= \bigcap_{q\cdot p>0} (E+q)$, there holds $T_{\widetilde E}=T_E$ (see \cite[Lem. 4.4]{bessiMassart}). Since $\widetilde E$ is recurrent, 
this implies that $T_E=T_u$ and therefore every plane-like minimizer with the Birkhoff property induces the same current $T_u$. 
\end{remark}

\end{document}